\documentclass{amsart}

\usepackage{amsmath,
            amssymb,
            bibentry,
            enumitem,
            graphicx,
            mathrsfs,
            tabu,
            tikz,
            xypic}
\usepackage[comma,numbers,sort,square]{natbib}
\usepackage[colorlinks=true,
            linktocpage=true,
            linkcolor=magenta,
            citecolor=magenta,
            urlcolor=magenta]{hyperref}
\PassOptionsToPackage{hyphens}{url}
\usetikzlibrary{arrows,shapes,snakes}

\definecolor{grey}{rgb}{0.95,0.95,0.95}
\definecolor{green}{rgb}{0.2,0.6,0.4}

\newtheorem{theorem}{Theorem}[section]
\newtheorem{proposition}[theorem]{Proposition}
\newtheorem{lemma}[theorem]{Lemma}
\newtheorem{corollary}[theorem]{Corollary}
\theoremstyle{definition}
\newtheorem{definition}[theorem]{Definition}

\newtheorem{question}[theorem]{Question}

\newtheorem*{note*}{Note}

\newtheoremstyle{principle}{}{}{\itshape}{}{\bfseries}{.}{.5em}{\thmnote{#3}#1}
\theoremstyle{principle}
\newtheorem*{principle}{}

\newcommand{\Psf}{\mathsf{P}}
\newcommand{\Qsf}{\mathsf{Q}}

\newcommand{\converges}{\mathord{\downarrow}}
\newcommand{\diverges}{\mathord{\uparrow}}

\newcommand{\Pb}{\mathbb{P}}

\newcommand{\dbf}{\mathbf{d}}

\newcommand{\Ccal}{\mathcal{C}}
\newcommand{\Dcal}{\mathcal{D}}

\newcommand{\Fcal}{\mathcal{F}}
\newcommand{\Gcal}{\mathcal{G}}

\newcommand{\Mcal}{\mathcal{M}}

\newcommand{\Rcal}{\mathcal{R}}
\newcommand{\Scal}{\mathcal{S}}

\newcommand{\uh}{{\upharpoonright}}

\renewcommand{\setminus}{\smallsetminus}

\newcommand{\s}[1]{\ensuremath{\sf{#1}}}

\DeclareMathOperator{\wkl}{\s{WKL}_0}

\DeclareMathOperator{\srt}{\s{SRT}}

\DeclareMathOperator{\coh}{\s{COH}}

\DeclareMathOperator{\semo}{\s{SEM}}

\DeclareMathOperator{\dhyp}{\s{DHYP}}
\DeclareMathOperator{\dhyt}{\s{DHYT}}

\usetikzlibrary{shapes,arrows}
\usetikzlibrary{decorations.markings}
\definecolor{lightblue}{HTML}{e6e6e6}
\definecolor{lightred}{HTML}{eca6a6}
\definecolor{lightgreen}{RGB}{164,244,140}

\newtheoremstyle{custom}
  {10pt}
  {10pt}
  {\normalfont}
  {}
  {\bfseries}
  {}
  { }
  {}

\theoremstyle{custom}

\usepackage{xcolor}	
\usepackage{soul}

\DeclareMathOperator{\res}{\upharpoonright}
\newcommand{\ran}{\operatorname{ran}}

\newcommand{\seq}[1]{\langle #1 \rangle}

\newcommand{\forces}{\Vdash}

\newcommand{\RCA}{\mathsf{RCA}}
\newcommand{\ACA}{\mathsf{ACA}}
\newcommand{\WKL}{\mathsf{WKL}}

\newcommand{\RT}{\mathsf{RT}}
\newcommand{\RCOH}{\mathsf{RCOH}}
\newcommand{\COH}{\mathsf{COH}}

\newcommand{\SRT}{\mathsf{SRT}}

\newcommand{\RWKLp}{\mathsf{RWKL}'}

\newcommand{\JI}{\mathsf{JI}}

\newcommand{\D}{\mathsf{D}}

\newcommand{\cequiv}{\equiv_{\text{\upshape c}}}
\newcommand{\scequiv}{\equiv_{\text{\upshape sc}}}
\newcommand{\uequiv}{\equiv_{\text{\upshape W}}}
\newcommand{\suequiv}{\equiv_{\text{\upshape sW}}}
\newcommand{\cred}{\leq_{\text{\upshape c}}}
\newcommand{\ncred}{\nleq_{\text{\upshape c}}}
\newcommand{\scred}{\leq_{\text{\upshape sc}}}
\newcommand{\nscred}{\nleq_{\text{\upshape sc}}}
\newcommand{\ured}{\leq_{\text{\upshape W}}}
\newcommand{\nured}{\nleq_{\text{\upshape W}}}
\newcommand{\sured}{\leq_{\text{\upshape sW}}}

\newcommand{\Tred}{\leq_{\text{\upshape T}}}
\newcommand{\nTred}{\nleq_{\text{\upshape T}}}

\renewcommand{\phi}{\varphi}
\renewcommand{\leq}{\leqslant}
\renewcommand{\geq}{\geqslant}
\renewcommand{\nleq}{\nleqslant}

\begin{document}

\title{Some results concerning the $\mathsf{SRT}^2_2$ vs.\ $\mathsf{COH}$ problem}

\author[Cholak]{Peter A.\ Cholak}
\address{Department of Mathematics\\
University of Notre Dame\\
255 Hurley Building, Notre Dame, Indiana 46556-4618 U.S.A.}
\email{ Peter.Cholak.1@nd.edu}

\author[Dzhafarov]{Damir D.\ Dzhafarov}
\address{Department of Mathematics\\
University of Connecticut\\
341 Mansfield Road\\ Storrs, Connecticut 06269-1009 U.S.A.}
\email{damir@math.uconn.edu}

\author[Hirschfeldt]{Denis R.\ Hirschfeldt}
\address{Department of Mathematics\\
University of Chicago\\
5734 S. University Ave., Chicago, Illinois 60637-1546 U.S.A.}
\email{drh@math.uchicago.edu}

\author[Patey]{Ludovic Patey}
\address{Institut Camille Jordan\\
Universit\'{e} Claude Bernard Lyon 1\\
43 boulevard du 11 novembre 1918, F-69622 Villeurbanne Cedex, France}
\email{ludovic.patey@computability.fr}

\begin{abstract}
	The $\SRT^2_2$ vs.\ $\COH$ problem is a central problem in computable combinatorics and reverse mathematics, asking whether every Turing ideal that satisfies the principle $\SRT^2_2$ also satisfies the principle $\COH$. This paper is a contribution towards further developing some of the main techniques involved in attacking this problem. We study several principles related to each of $\SRT^2_2$ and $\COH$, and prove results that highlight the limits of our current understanding, but also point to new directions ripe for further exploration.
\end{abstract}

\thanks{Cholak was partially supported by a grant from the Simons Foundation (\#315283). Dzhafarov was supported by grant DMS-1400267 from the National Science Foundation of the United States and a Collaboration Grant for Mathematicians from the Simons Foundation. Hirschfeldt was partially supported by grant DMS-1600543 from the National Science Foundation of the United States. Patey was partially supported by grant ANR ``ACTC'' \#ANR-19-CE48-0012-01. Cholak, Dzhafarov, and Hirschfeldt were also partially supported by a Focused Research Group grant from the National Science Foundation of the United States, DMS-1854136, DMS-1854355, and DMS-1854279, respectively. All authors thank the Mathematisches Forschungsinstitut Oberwolfach, for hosting them as part of its Research in Pairs Program during the fall of 2016 and the winter of 2018, during which the work presented here was conducted. Additional thanks go to the two anonymous referees for their valuable comments and suggestions.}

\maketitle

\section{Introduction}

One of the most fruitful programs of research in computability theory over the past few decades has been the investigation of the logical strength of combinatorial principles, particularly Ramsey's theorem and its many relatives.
Ramsey's theorem is, of course, a far-reaching result, broadly asserting that in any configuration of objects, some amount of order is necessary. Understanding this order has been the objective of much research in combinatorics and logic. In computability theory, and even more so reverse mathematics, it has spawned a long and productive line of research. See Hirschfeldt~\cite[Chapter 6]{Hirschfeldt-2014} for an introduction.

For many years, a central problem surrounding this analysis has been to clarify the relationship between two important variants of Ramsey's theorem for pairs; specifically, whether the
Cohesiveness principle ($\COH$) is implied by the stable Ramsey's theorem for pairs ($\SRT^2_2$) over the weak fragment $\RCA_0$ of second-order arithmetic. This question was finally answered in 2014 by Chong, Slaman, and Yang~\cite{CSY-2014}, who gave a negative answer, but remarkably, using a nonstandard model for the separation of these principles. What has come to be called the \emph{$\SRT^2_2$ vs.\ $\COH$ problem} is the question of what happens in $\omega$-models (models with standard first-order part), and this question remains open. Over the past several years, work on the effective content of variants of Ramsey's theorem, including towards a solution of the above problem, has driven much of the progress in computable combinatorics and the reverse mathematics of combinatorial principles. It has also been an important impetus for the fruitful and growing intersection of these subjects with computable analysis (see, e.g.,~\cite{ADSS-2017, BR-2017, DDHMS-2016, Dzhafarov-2016, DPSW-2017, HJ-2016, HM-TA, HM-2017, MP-TA, Nichols-TA, Patey-2016c, Patey-2016}).


In this paper, we study several aspects of this problem. We assume familiarity with computability theory and reverse mathematics, and refer the reader to Soare~\cite{Soare-2016} and Simpson~\cite{Simpson-2009}, respectively, for background. We also refer to Brattka, Gherardi, and Pauly~\cite{BGP-TA} for a survey of Weihrauch reducibility and computable analysis, though we include a brief summary of the most relevant concepts below.

\begin{note*}
	Since the submission of this article, Monin and Patey \cite{MP-TA2} have announced a solution to the $\SRT^2_2$ vs.\ $\COH$ problem, exhibiting an $\omega$-model in which $\SRT^2_2$ holds but $\COH$ fails. Their proof proceeds by entirely different methods than those explored within this paper. Hence, it does not supersede any of the results below, which we feel are of independent interest concerning the degree-theoretic content of the $\SRT^2_2$ and $\COH$ problems. In addition, Monin and Patey's result still leaves open Question \ref{Q:omniRT}, which may be regarded as the purely combinatorial variant of the $\SRT^2_2$ vs.\ $\COH$ problem, as well as the open questions in Section \ref{S:questions}.
\end{note*}

Throughout, we will be dealing with $\Pi^1_2$ statements of second-order arithmetic, which are examples of the more general concept of a \emph{problem}, as defined below.

\begin{definition}
\
	\begin{enumerate}
		\item A \emph{problem} $\mathsf{P}$ is a subset of $2^\omega \times 2^\omega$.
		\item Each $X \in 2^\omega$ for which there is a $Y \in 2^\omega$ such that $\mathsf{P}(X,Y)$ holds is an \emph{instance of $\mathsf{P}$}.
		\item Each $Y$ such that $\mathsf{P}(X,Y)$ holds is a \emph{solution to $X$ as an instance of $\mathsf{P}$}.
	\end{enumerate}
\end{definition}

\noindent When no confusion can arise, we will speak just of instances and solutions, without explicitly referencing the problem. When necessary, we may call an instance of $\mathsf{P}$ a \emph{$\mathsf{P}$-instance} for short, and a solution to some instance of $\mathsf{P}$ a \emph{$\mathsf{P}$-solution}.

Throughout, we follow the standard practice of coding mathematical objects and structures by numbers and sets of numbers. This makes our definition of problem very broad, since it permits us to deal with any instances and solutions that admit some kind of countable presentation. For all the objects we consider here, these codings will be obvious and/or well understood, so we will do so implicitly and informally. But we refer the reader to~\cite[Remarks 1.4 and 1.5]{HJ-2016} for a more thorough discussion of this issue, along with some explicit examples.

All of the $\Pi^1_2$ principles we will look at can be naturally put into the syntactic form
\[
	(\forall X)[\phi(X) \to (\exists Y)[\theta(X,Y)]],
\]
where $\phi$ and $\theta$ are arithmetical formulas. We can then view such a principle as a problem in the above sense, with the instances being all the $X \in 2^\omega$ such that $\phi(X)$ holds, and the solutions to any such $X$ being all the $Y \in 2^\omega$ such that $\psi(X,Y)$ holds. We shall make this identification without further mention in the sequel.

To compare problems, we will employ the following reductions.

\begin{definition}\label{D:reductions}
	Let $\mathsf{P}$ and $\mathsf{Q}$ be problems.
	\begin{enumerate}
		\item $\mathsf{P}$ is \emph{computably reducible} to $\mathsf{Q}$, written $\mathsf{P} \cred \mathsf{Q}$, if every instance $X$ of $\mathsf{P}$ computes an instance $\widehat{X}$ of $\mathsf{Q}$, such that for every solution $\widehat{Y}$ to $\widehat{X}$, we have that $X \oplus \widehat{Y}$ computes a solution $Y$ to $X$.
		\item $\mathsf{P}$ is \emph{strongly computably reducible} to $\mathsf{Q}$, written $\mathsf{P} \scred \mathsf{Q}$, if every instance $X$ of $\mathsf{P}$ computes an instance $\widehat{X}$ of $\mathsf{Q}$, such that every solution $\widehat{Y}$ to $\widehat{X}$ computes a solution $Y$ to $X$.
		\item $\mathsf{P}$ is \emph{Weihrauch reducible} to $\mathsf{Q}$, written $\mathsf{P} \ured \mathsf{Q}$, if there exist Turing functionals $\Phi$ and $\Psi$ such that for every instance $X$ of $\mathsf{P}$, we have that $\Phi^X$ is an instance of $\mathsf{Q}$, and for every solution $\widehat{Y}$ to $\Phi^X$ we have that $\Psi^{X \oplus \widehat{Y}}$ is a solution to $X$.
		\item $\mathsf{P}$ is \emph{strongly Weihrauch reducible} to $\mathsf{Q}$, written $\mathsf{P} \sured \mathsf{Q}$, if there exist Turing functionals $\Phi$ and $\Psi$ such that for every instance $X$ of $\mathsf{P}$, we have that $\Phi^X$ is an instance of $\mathsf{Q}$, and for every solution $\widehat{Y}$ to $\Phi^X$ we have that $\Psi^{\widehat{Y}}$ is a solution to $X$.
	\end{enumerate}
\end{definition}

\noindent We say $\mathsf{P}$ and $\mathsf{Q}$ are \emph{computably equivalent}, and write $\mathsf{P} \cequiv \mathsf{Q}$, if $\mathsf{P} \cred \mathsf{Q}$ and $\mathsf{Q} \cred \mathsf{P}$. We analogously define \emph{strong computable equivalence}, \emph{Weihrauch equivalence}, and \emph{strong Weihrauch equivalence}, denoted by $\scequiv$, $\uequiv$, and $\suequiv$, respectively.

The relationships between these reductions are easy to see, and are summarized in Figure~\ref{F:reductionsrelations}. Weihrauch reducibility was introduced by Weihrauch~\cite{Weihrauch-1992}, and computable reducibility by Dzhafarov~\cite{Dzhafarov-2015}. The connection with reverse mathematics comes from the fact that all of these reducibilities are stronger than implication over $\omega$-models for $\Pi^1_2$ principles. That is, if $\mathsf{P}$ and $\mathsf{Q}$ are $\Pi^1_2$ principles and (when viewed as problems) $\mathsf{P}$ is reducible to $\mathsf{Q}$ in any of the senses above, then every $\omega$-model of $\mathsf{Q}$ also satisfies $\mathsf{P}$. And while implication over $\omega$-models is a strictly more general notion, it is a well-known empirical fact that most such implications found in the literature are due to one of the stronger reducibilities above. For a more thorough discussion of this phenomenon, see Hirschfeldt and Jockusch~\cite[Section 4.1]{HJ-2016}.

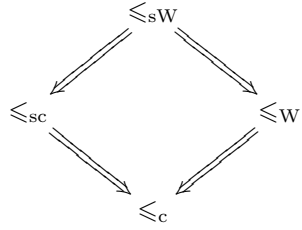
\begin{figure}\label{F:reductionsrelations}
\[
\xymatrix{
& \sured \ar@2[dl] \ar@2[dr] \\
\scred \ar@2[dr] & & \ured \ar@2[dl] \\
& \cred
}
\]
\caption[]{Relations between notions of reduction. An arrow from one reducibility to another means that whenever $\mathsf{Q}$ is reducible to $\mathsf{P}$ according to the first then it is also reducible according to the second. In general, no relations hold other than the ones shown.}
\end{figure}

To state the $\SRT^2_2$ vs.\ $\COH$ problem, we now review some standard definitions from Ramsey theory.

\begin{definition}
	Fix a set $X \subseteq \omega$, and integers $n,k \geq 1$.
	\begin{enumerate}
		\item $[X]^n$ denotes the set $\{ \seq{x_0,\ldots,x_{n-1}} \in X^n : x_0 < \cdots < x_{n-1}\}$.
		\item A \emph{$k$-coloring of $[X]^n$} is a map $c : [X]^n \to \{0,\ldots,k-1\}$.
		\item A set $Y \subseteq X$ is a \emph{homogeneous set} for $c$ if $c \res [Y]^n$ is constant.
	\end{enumerate}
\end{definition}

\noindent As the number of colors typically will not matter for our purposes, we shall usually speak only of \emph{colorings}, rather than explicitly about \emph{$k$-colorings} for a given $k$. Except in the statements of definitions, we will usually be working with $k = 2$ anyway. We abbreviate colorings $c : [X]^n \to \{0,\ldots,k-1\}$ by $c : [X]^n \to k$, as usual, and for $\seq{x_0,\ldots,x_{n-1}} \in [X]^n$ we write $c(x_0,\ldots,x_{n-1})$ in place of $c(\seq{x_0,\ldots,x_{n-1}})$.

\begin{principle}[Ramsey's theorem ($\RT^n_k$)]
	Every coloring $c : [\omega]^n \to k$ has an infinite homogeneous set.
\end{principle}

Of particular interest in computability has been Ramsey's theorem for pairs, i.e., $\RT^2_2$. The principles $\SRT^2_2$ and $\COH$ come from a prominent approach, pioneered by Cholak, Jockusch, and Slaman~\cite{CJS-2001}, of splitting combinatorial principles into a stable and a cohesive half.

\begin{definition}
\
	\begin{enumerate}
		\item A coloring $c : [\omega]^2 \to k$ is \emph{stable} if $\lim_y c(x,y)$ exists for every $x \in \omega$.
		\item An infinite set $L \subseteq \omega$ is \emph{limit-homogeneous} for such a $c$ if there is an $i < k$ such that $\lim_y c(x,y) = i$ for all $x \in L$.
	\end{enumerate}
\end{definition}

\noindent The stable form of $\RT^2_k$ takes two natural forms.

\begin{principle}[Stable Ramsey's theorem for pairs ($\SRT^2_k$)]
	Every stable $c : [\omega]^2 \to k$ has an infinite homogeneous set.
\end{principle}

\begin{principle}[$\Delta^0_2$ subset principle ($\D^2_k$)]
	Every stable $c : [\omega]^2 \to k$ has an infinite limit-homogeneous set.
\end{principle}

\noindent The name of the second principle derives from the observation that, by the limit lemma, computing an infinite limit-homogeneous set for a given computable stable 2-coloring is exactly the same as computing an infinite subset of a given $\Delta^0_2$ set or its complement. It is well-known that $\SRT^2_2$ and $\D^2_2$ are equivalent over $\RCA_0$. This equivalence is easy to see for $\omega$-models (see, e.g.,~\cite[Lemma 3.5]{CJS-2001}, which in fact gives a computable equivalence), but requires a delicate argument, due to Chong, Lempp, and Yang~\cite{CLY-2010}, to formalize with limited induction. Dzhafarov~\cite[Corollary 3.3 and Corollary 3.6]{Dzhafarov-2016} showed that $\SRT^2_2 \nured \D^2_2$ and $\SRT^2_2 \nscred \D^2_2$. As discussed further below, $\D^2_2$ is just a less effective version of $\RT^1_2$. In practice, this often makes $\D^2_2$ easier to work with than $\SRT^2_2$.

For sets $X$ and $Y$, let $X \subseteq^* Y$ denote that there is a finite set $F$ such that $X \setminus F \subseteq Y$.


\begin{definition}
	Let $\vec{R} = \seq{R_0, R_1, \dots}$ be a sequence of sets. A set $C$ is \emph{cohesive} for $\vec{R}$ if for every $n \in \omega$, either $C \subseteq^{*} R_n$ or $C \subseteq^{*} \overline{R}_n$.
\end{definition}

\begin{principle}[Cohesiveness principle ($\COH$)]
	Every sequence of sets admits an infinite cohesive set.	
\end{principle}

The relevant fact for us, due to Cholak, Jockusch, and Slaman~\cite[Lemma 7.11]{CJS-2001}, with the use of $\Sigma^0_2$-induction later eliminated by Mileti~\cite[Claim A.1.3]{Mileti-2004} and Jockusch and Lempp (unpublished), is that $\RT^2_2$ is equivalent over $\RCA_0$ to the conjunction $\SRT^2_2 + \COH$. Each of $\SRT^2_2$ and $\COH$ is combinatorially simpler than $\RT^2_2$ in a number of ways. Part of this simplicity comes from the fact that both principles can be viewed in terms of the more elementary principle $\RT^1_2$. For $\SRT^2_2$, in the form $\D^2_2$, this is because finding a limit-homogeneous set for a stable coloring $c : [\omega]^2 \to 2$ is the same as finding an infinite homogeneous set for the coloring $d : \omega \to 2$ defined by $d(x) = \lim_y c(x,y)$. Note, by the way, that $d$ is $c'$-computable, so $\D^2_2$ can be characterized as the jump of $\RT^1_2$. For $\COH$, the connection with $\RT^1_2$ follows by a result of Jockusch and Stephan~\cite[Theorem 2.1]{JS-1993}, who characterized the degrees containing an infinite cohesive set for every computable family of sets as precisely those degrees $\mathbf{a}$ satisfying $\mathbf{a}' \gg \mathbf{0}'$. Relativizing this result, it is easy to see that $\COH$ is computably equivalent to the assertion that given a sequence $\seq{d_0, d_1, \ldots}$ of colorings $\omega \to 2$, there exists a sequence $\seq{H_0, H_1, \ldots}$ of infinite sets such that each $H_n$ is, up to finite error, homogeneous for $d_n$. In the parlance of computable analysis, this says that $\COH$ is computably equivalent to the parallelization of the principle $(\RT^1_2)^{\mathrm{fe}}$ asserting that for every coloring of singletons there is an infinite set that is homogeneous modulo finitely many elements.

We can now formally state the main problem we are interested in.

\begin{principle}[The $\SRT^2_2$ vs.\ $\COH$ problem]
	Does every $\omega$-model of $\SRT^2_2$ satisfy $\COH$?
\end{principle}

\noindent Conventional wisdom suggests the answer ought to be negative, since implications between relatively straightforward combinatorial principles are usually quite elementary. The only possible difficulties one expects to encounter are induction issues, but these are precisely the ones that are absent when working over $\omega$-models. On the other hand, the continued resistance of this problem to a separation, in spite of a string of recent advances that did confirm other long-conjectured non-implications (e.g., Liu~\cite{Liu-2012}, and Lerman, Solomon, and Towsner~\cite{LST-2013}), means we should probably keep an open mind.

Our paper is a contribution to the study of this problem. In Section~\ref{S:weakening} we introduce a weaker form of $\COH$, and show it to be a combinatorial consequence of $\SRT^2_2$. This is a step towards resolving a question of Patey~\cite[Question 2.10]{Patey-2016d}, as well as  the longstanding question of whether $\COH$ is computably reducible to $\SRT^2_2$ (see~\cite[Question 5.3]{HJ-2016}). In Sections~\ref{S:locks} and~\ref{S:symmetric}, we prove results about the complexities of instances of $\COH$ and solutions to $\SRT^2_2$, respectively. The aim is to identify the precise features of cohesiveness and homogeneity that might be responsible for a separation or implication, as the case may be. Section~\ref{S:symmetric} also presents a new method of constructing effective solutions to $\SRT^2_2$/$\D^2_2$ that we hope will find further applications. In Section \ref{S:hyperimmunity}, we study variants of hyperimmunity, to better understand the class of instances of $\SRT^2_2$ having solutions that do not compute cohesive sets. Finally, in Section~\ref{S:questions}, we lay out some additional questions and directions for future research related to the $\SRT^2_2$ vs.\ $\COH$ problem.

\section{Weakening COH}\label{S:weakening}

One way of attacking the $\SRT^2_2$ vs.\ $\COH$ problem has been by showing that $\COH$ is at least not reducible to $\SRT^2_2$ in a typical way, i.e., via any of the notions in Definition~\ref{D:reductions}. Such results lend credence to a negative answer, since a full $\omega$-model separation of $\COH$ from $\SRT^2_2$ would in particular yield all such non-reductions. Recent examples along these lines include work by Dzhafarov~\cite{Dzhafarov-2015, Dzhafarov-2016}, Patey~\cite{Patey-2016}, and Dzhafarov, Patey, Solomon, and Westrick~\cite{DPSW-2017}, establishing, among other results, that $\COH \nured \SRT^2_2$ and $\COH \nscred \SRT^2_2$. As already remarked, it remains open whether $\COH \cred \SRT^2_2$.

In showing that some principle $\mathsf{P}$ is not, say, computably or Weihrauch reducible to some other principle $\mathsf{Q}$, we must exhibit an instance $X$ of $\mathsf{P}$, and for each instance of $\mathsf{Q}$ \emph{computable from $X$}, we must exhibit one solution against which to diagonalize. But in some cases, we can in fact do this for \emph{all} instances of $\mathsf{Q}$, whether computable from $X$ or not. The first results along these lines were obtained by Hirschfeldt and Jockusch~\cite[proofs of Lemma 3.2 and Theorem 3.3]{HJ-2016} and Patey~\cite[Theorem 3.2]{Patey-2016}. Conversely, even when we know that $\mathsf{P}$ is not reducible to $\mathsf{Q}$ according to any of the notions in Definition~\ref{D:reductions}, it may still be that $\mathsf{P}$ is reducible to $\mathsf{Q}$ in this stronger sense, where the instances of $\mathsf{Q}$ are allowed to be arbitrary. Intuitively, we can think of $\mathsf{P}$ as being a \emph{combinatorial consequence} of $\mathsf{Q}$. This notion was first isolated and studied by Monin and Patey~\cite{MP-TA} under the name of \emph{omniscient computable reducibility}.

\begin{definition}[\cite{MP-TA}, Section 1.1]
	Let $\mathsf{P}$ and $\mathsf{Q}$ be problems. Then $\mathsf{P}$ is \emph{omnisciently computably reducible} to $\mathsf{Q}$ if for every instance $X$ of $\mathsf{P}$ there exists an instance $\widehat{X}$ of $\mathsf{Q}$, such that for every solution $\widehat{Y}$ to $\widehat{X}$ we have that $X \oplus \widehat{Y}$ computes a solution $Y$ to $X$.
\end{definition}

We start with the following relatively straightforward result, which nicely illustrates the power of this reducibility. As we will see, this is also a very insightful example for studying the $\SRT^2_2$ vs.\ $\COH$ problem. For notational convenience, given any set $R$ we write $R^0$ for $R$ and $R^1$ for $\overline{R}$. Given a family of sets $\vec{R} = \seq{R_0,R_1,\ldots}$ and a finite string $\sigma \in 2^{<\omega}$, we also write
\[
	R^\sigma = \bigcap_{n < |\sigma|} R^{\sigma(n)}_n.
\]

\begin{proposition}\label{P:omniCOHtoSRT}
	$\COH$ is omnisciently computably reducible to $\SRT^2_2$.
\end{proposition}

\begin{proof}
	Let $\vec{R} = \seq{R_0, R_1, \ldots}$ be a sequence of sets. Define $c : [\omega]^2 \to 2$ as follows: for all $x < y$,
	\[
	c(x,y) =
	\begin{cases}
		0 & \text{if } (\exists \sigma \in 2^x)(\exists z > y)[z \in R^\sigma \text{ and } R^\sigma \text{ is finite}],\\ 
		1 & \text{otherwise}.
	\end{cases}
	\]
	Given $x$, let $m_x > x$ be least such that $\max R^\sigma \leq m_x$ for all $\sigma \in 2^{x+1}$ for which $R^\sigma$ is finite. Then $c(x,y) = 1$ for all $y \geq m_x$, and $c(x,y) = 0$ for all $y$ with $x < y < m_x$. In particular, $\lim_y c(x,y) = 1$ for all $x$, so $c$ is stable, and every infinite homogeneous set for $c$ must have color $1$. Furthermore, if $H = \{h_0 < h_1 < \cdots\}$ is any such homogeneous set, then necessarily $m_x \leq m_{h_x} \leq h_{x+1}$ for all $x$.
	
	We can thus use $\vec{R} \oplus H$ to compute a sequence of binary strings $\sigma_0 \prec \sigma_1 \prec \cdots$ (ordered by extension) such that $|\sigma_x| = x$ and $R^{\sigma_x}$ is infinite. Let $\sigma_0 = \emptyset$, and suppose by induction that we have defined $\sigma_x$ for some $x \geq 0$ and that $R^{\sigma_x}$ is infinite. Since either $R^{\sigma_x} \cap R^0_x$ or $R^{\sigma_x} \cap R^1_x$ is infinite, there must be a $z > h_{x+1}$ in one of these two sets. Search for the least such $z$, and let $\sigma_{x+1} = \sigma_x b$ for whichever $b \in \{0,1\}$ has $z \in R^{\sigma_x} \cap R^b_x$. Since $z \geq m_x$, we know that $R^{\sigma_{x+1}}$ is infinite, as desired.
	
	Finally, from $\sigma_0 \prec \sigma_1 \prec \cdots$ we can $\vec{R}$-computably define an infinite cohesive set $C$ for $\vec{R}$ in the standard way. For completeness, we give the details. Given $x \in \omega$, and having defined integers $d_y$ for all $y < x$, we let $d_x$ be the least element of $R_{\sigma_{x+1}}$ larger than all these $d_y$. We then let $C = \{d_0 < d_1 < \cdots\}$. To see that this set is cohesive for $\vec{R}$, consider any $y \in \omega$ and let $b = \sigma_{y+1}(y)$. Then for each $x \geq y$ we have that $d_x \in R^{\sigma_{x+1}(y)}_y = R^b_y$, so $C \subseteq^* R^b_y$.
\end{proof}

In the above proof, the constructed coloring $c$ has complicated homogeneous sets essentially because they must all be very sparse, which allows us to code in jump information. Computing the jump is enough to produce cohesiveness, but it is in fact much stronger. Thus, coding using sparseness is not helpful for understanding the true relationship between $\SRT^2_2$ and $\COH$. A better approach might be through limit-homogeneity, where sparseness is not so easily forced. Indeed, notice that $c$ above has very uncomplicated limit-homogeneous sets: in particular, $\omega$ is limit-homogeneous for $c$. This observation prompted Patey~\cite{Patey-2016d} to ask whether Proposition~\ref{P:omniCOHtoSRT} still holds if $\SRT^2_2$ is replaced by $\D^2_2$. Notice that this is equivalent to replacing $\SRT^2_2$ by $\RT^1_2$, since as pointed out above, $\D^2_2$ is just $\RT^1_2$ with  instances given by limit approximations, and thus the two are omnisciently computably equivalent.

\begin{question}[{\cite[Question 2.10]{Patey-2016d}}]\label{Q:omniRT}
  Is~$\COH$ omnisciently computably reducible to~$\RT^1_2$?
\end{question}

As a way to show how, in principle, cohesiveness can be coded into the homogeneous sets of colorings of singletons, we introduce the following weakening of $\COH$, and show it to be a combinatorial consequence of $\RT^1_2$.

\begin{definition}
	Let $\vec{R} = \seq{R_0, R_1, \dots}$ be a sequence of sets. A set $C$ is \emph{Ramsey-cohesive} for $\vec{R}$ if for infinitely many $n \in \omega$, either $C \subseteq^{*} R_n$ or $C \subseteq^{*} \overline{R}_n$.
\end{definition}


\begin{principle}[Ramsey-type cohesiveness principle ($\RCOH$)]
	Every sequence of sets admits an infinite Ramsey-cohesive set.
\end{principle}

\begin{theorem}
	$\RCOH$ is omnisciently computably reducible to $\RT^1_2$.
\end{theorem}

\begin{proof}
	Given an instance $\vec{R} = \seq{R_0,R_1,\ldots}$ of $\RCOH$, we define $c : \omega \to 2$ inductively. Fix $n \in \omega$, and suppose we have defined $c \res n$, which we regard as an element of $2^{<\omega}$ of length $n$. We let $c(n) = 0$ if $R_n \cap R^{c \res n}$ is infinite, and otherwise we let $c(n) = 1$. By induction, it follows that $R^{c \res n}$ is infinite for each $n$. Now let $H$ be any infinite homogeneous set for $c$, say with color $i < 2$. Then, in particular, for each $n$ we have that
	\[
		\bigcap_{m \in H \res n} R^i_m \supseteq R^{c \res n},
	\]
	so the intersection on the left is infinite. We can thus compute an infinite Ramsey-cohesive set $C$ for $\vec{R}$ from $\vec{R} \oplus H$, as follows: having defined $C \res n$ for some $n \in \omega$, choose the least element of $\bigcap_{m \in H \res n+1} R^i_m$ larger than $\max (C \res n)$, and add it to $C$. Clearly, $C$ is infinite, and for each $m \in H$ we have $C \subseteq^* R^i_m$.
\end{proof}

Obviously, $\RCOH$ is a restriction, and hence a logical consequence, of $\COH$. Unfortunately, it is also strictly weaker than $\COH$, which we show next, after a few auxiliary lemmas. Thus, the above result does not settle Question~\ref{Q:omniRT}.

\begin{definition}\label{def:cohesive-decision}
	Let $\vec{R} = \seq{R_0, R_1,\ldots}$ be a sequence of sets. We let
	\[
		\Ccal(\vec{R}) = \{ P \in 2^\omega : (\forall \sigma \prec P)[R^\sigma \mbox{ is infinite}]\}.
	\]
\end{definition}

Note that if $\vec{R}$ is computable, then $\Ccal(\vec{R})$ is a $\Pi^{0,\emptyset'}_1$ class.

\begin{lemma}[{Patey~\cite[Lemma 2.4]{Patey-2016}}]\label{lem:tree-to-sequence}
	For every $\Delta^0_2$ infinite tree $T \subseteq 2^{<\omega}$, there is a computable sequence of sets $\vec{R}$ such that $\Ccal(\vec{R}) = [T]$.
\end{lemma}

\begin{definition}
	Let $T \subseteq 2^{<\omega}$ be an infinite tree.
	\begin{enumerate}
		\item A set $H$ is \emph{homogeneous} for a string $\sigma \in 2^{<\omega}$ if there is some $i < 2$ such that $(\forall x < |\sigma|)[ x \in H \to \sigma(x) = i]$.
		\item A set $H$ is \emph{homogeneous} for $T$ if the set $\{ \sigma \in T : H \mbox{ is homogeneous for } \sigma \}$ is infinite.
	\end{enumerate}
\end{definition}

\begin{lemma}\label{lem:degrees-of-Ramsey-cohesives}
	Let $T \subseteq 2^{<\omega}$ be an infinite $\Delta^0_2$ tree, and let $\vec{R} = \seq{R_0, R_1,\ldots}$ be a computable sequence of sets such that $\Ccal(\vec{R}) = [T]$. Then the sets computing an infinite Ramsey-cohesive set for $\vec{R}$ are exactly those whose jumps compute infinite homogeneous sets for~$T$.
\end{lemma}
\begin{proof}
Let $C$ be a Ramsey-cohesive set for $\vec{R}$. 
The sets $U = \{ i : C \subseteq^{*} R_i \}$ and $V = \{ i : C \subseteq^{*} \overline{R}_i \}$
are $\Sigma^{0,C}_2$, and one of them is infinite. Say $U$ is infinite, the other case being symmetric.
Then $U$ has an infinite $\Delta^{0,C}_2$ subset $U_1$. The set $U_1$ is homogeneous for $T$.

Conversely, let $H$ be an infinite set homogeneous for~$T$, say for color 1, and let $f : \omega^2 \to 2$ be a stable function such that $\lim_y f(x, y) = H(x)$. Define the set $C = \{ x_0 < x_1 < \cdots \}$ $f$-computably as follows. First, let $x_0 = 0$. Then, having defined $x_n$, search for some stage $s > n$  and some $x_{n+1} \in \bigcap_{i < n, f(i,s) = 1} R_i$ such that $x_{n+1} > x_n$. Such $s$ and $x_{n+1}$ must be found. Indeed, let $s$ be large enough so that $(\forall i < n)f(i,s) = \lim_y f(i,y)$. Then the set $F = \{ f(i,s) : i < n \}$ is a subset of $H$, hence is homogeneous for $T$ with color 1. Since $[T] = \Ccal(\vec{R})$, there is some $P \in  \Ccal(\vec{R})$ such that $F \subseteq P$. In particular, $\bigcap_{i < n, f(i,s) = 1} R_i = \bigcap_{i \in F} R_i \supseteq R_{P \uh \max F}$ is infinite, so there is some $x_{n+1} > x_n$ in it. It is easy to check that $C$ is Ramsey-cohesive for $\vec{R}$.
\end{proof}

The following principle was introduced by Flood~\cite{Flood-2012}.

\begin{principle}[Ramsey-type weak K\"{o}nig's lemma ($\mathsf{RWKL}$)]
	Every infinite binary tree has an infinite homogeneous set.
\end{principle}


Our interest below will be in a relativized form of the above principle. Namely, we will look at $\RWKLp$, which is the assertion that for every function $g : 2^{<\omega} \times \omega \to 2$ such that $\lim_s g(\sigma,s)$ exists for all $\sigma$ and $T = \{ \sigma : \lim_s g(\sigma,s) = 1\}$ is a tree, there exists an infinite homogeneous set for $T$. Let $\JI$ be the problem whose instances are all $X \in \omega^\omega$, and the solutions to any such $X$ are all functions $f : \omega^2 \to \omega$ such that $(\forall x)~X(x) = \lim_s f(x, s)$. Below, $\JI \circ \RWKLp$ denotes the problem whose instances are all $\RWKLp$-instances, and given any such instance $g$, a $\JI \circ \RWKLp$-solution to it is any $\JI$-solution to any $\RWKLp$-solution to $g$, i.e., a function $f : \omega^2 \to 2$ such that $\lim_s f(x,s)$ exists for all $x$ and $\{x : \lim_s f(x,s) = 1\}$ is an infinite homogeneous set for the tree $\{ \sigma : \lim_s g(\sigma,s) = 1\}$.


\begin{lemma}
$\JI \circ \RWKLp \cequiv \RCOH$.
\end{lemma}
\begin{proof}
To see that $\JI \circ \RWKLp \cred \RCOH$, let $g : \omega^2 \to 2^{<\omega}$ be a stable function
such that $T = \{ \sigma \in 2^{<\omega} : \lim_s g(\sigma, s) = 1\}$ is a binary tree.
By Lemma~\ref{lem:tree-to-sequence}, there is a $g$-computable sequence of sets $\vec{R}$
such that $[T] = \Ccal(\vec{R})$. Let $C$ be a Ramsey-cohesive set for $\vec{R}$. By Lemma~\ref{lem:degrees-of-Ramsey-cohesives}, there is a $g \oplus C$-computable function 
$f : \omega^2 \to 2$ such that the set $H = \{ x : \lim_s f(x,y) = 1 \}$
is infinite and homogeneous for $T$.
In particular, $f$ is a solution to $T$ viewed as an instance of $\JI \circ \RWKLp$.

Now to see that $\RCOH \cred \JI \circ \RWKLp$, let $\vec{R}$ be an instance of $\RCOH$. Then there is a $\Delta^{0,\vec{R}}_2$ tree $T$ such that $[T] = \Ccal(\vec{R})$.
Let $f : \omega^2 \to 2$ be a function such that the set $H = \{ x : \lim_s f(x,y) = 1 \}$
is infinite and homogeneous for $T$. By Lemma~\ref{lem:degrees-of-Ramsey-cohesives},
$f \oplus \vec{R}$ computes a Ramsey-cohesive set $C$ for $\vec{R}$.
\end{proof}

%
%

\begin{theorem}
There is an $\omega$-model of $\RCOH$ that is not a model of $\COH$.
\end{theorem}
\begin{proof}
We build an increasing sequence of sets $X_0 \Tred X_1 \Tred \cdots$
such that for every $i$, the jump of $X_i$ is not of PA degree relative to $\bf 0'$,
and for every $i$ and every $X_i$-computable sequence of sets $\vec{R}$,
there is a $j$ such that $X_j$ computes an infinite Ramsey-cohesive set for $\vec{R}$.  The Turing ideal $\Mcal = \{ Z : (\exists i)~Z \Tred X_i \}$ is then a model of $\RCOH$, but not a model of $\COH$. Indeed, the instance of $\COH$ composed of all the primitive recursive sets belongs to $\Mcal$, but by Jockusch and Stephan~\cite[Theorem 2.1]{JS-1993}, any cohesive set for this sequence has jump of PA degree over $\bf 0'$.

Start with $X_0 = \emptyset$. Having defined $X_0, \dots, X_i$,
let $\vec{R}$ be the next $X_i$-computable sequence of sets in an order chosen so that we eventually consider every sequence of sets computable in any $X_j$.
By Lemmas~\ref{lem:tree-to-sequence} and~\ref{lem:degrees-of-Ramsey-cohesives}, there is an $\vec{R}$-computable instance $f$ of $\RWKLp$ 
such that for every set $Y$ whose jump computes a solution to $f$, we have that $Y \oplus \vec{R}$ computes an infinite Ramsey-cohesive set for $\vec{R}$.
Let $P$ be a path through the tree $T = \{ \sigma : \lim_y f(\sigma, y) = 1 \}$. Since $X_i'$ is not of PA degree relative to $\bf 0'$ by inductive hypothesis, it follows by Liu's theorem~\cite[Theorem 1.5]{Liu-2012} that there is an infinite set $H \subseteq P$ or $H \subseteq \overline{P}$ such that $H \oplus X_i'$ is not of PA degree over $\bf 0'$. By the relativized Friedberg jump inversion theorem (see~\cite[Theorem 2.16.1]{DH-2010}), there is a set $X_{i+1} \geq_T X_i$ such that $X_{i+1}' \equiv_T H \oplus X_i'$. In particular, $X_{i+1}$ computes an infinite Ramsey-cohesive set for $\vec{R}$ and $X_{i+1}'$ is not of PA degree over $\bf 0'$.
\end{proof}

\section{$\COH$ and Cohen forcing with locks}\label{S:locks}

An important feature of $\COH$ is that it has a \emph{universal instance}. More precisely, for each set $A$ there is an $A$-computable instance $\vec{R} = \seq{R_0,R_1,\ldots}$ of $\COH$ such that for any other $A$-computable instance $\vec{S} = \seq{S_0,S_1,\ldots}$ and any solution $C$ to $\vec{R}$, we have that $A \oplus C$ computes a solution to $\vec{S}$. This fact follows by the aforementioned result of Jockusch and Stephan~\cite[Theorem 2.1]{JS-1993}, which can be restated as follows: for any set $X$, there is an $A \oplus X$-computable solution to every $A$-computable instance of $\COH$ precisely when $\deg(A \oplus X)' \gg \deg(A)'$. Thus, an $A$-computable instance $\vec{R}$ of $\COH$ is universal
just in case every solution $C$ to $\vec{R}$ satisfies $\deg(A \oplus C)' \gg \deg(A)'$. As first pointed out in~\cite{JS-1993}, the sequence of primitive $A$-recursive sets has this property.

The existence of universal instances means that, in principle, there is no need to ever \emph{construct} complicated instances of $\COH$---such as for separation or non-reduction results---since there exists a maximally complicated one. In practice, however, constructing the instance explicitly often gives more flexibility. This is done, for example, by Dzhafarov~\cite[Theorem 5.2]{Dzhafarov-2016} to prove that $\coh \nured \srt^2_2$, and by Dzhafarov, Patey, Solomon, and Westrick~\cite[Corollary 1.6]{DPSW-2017} to prove that $\COH \nscred \srt^2_{<\infty}$. Both of these arguments use a forcing notion, sometimes called \emph{Cohen forcing with locks}, which is very natural for building instances of $\COH$. However, we show in this section that this forcing does not produce \emph{maximally} complicated (i.e., universal) instances. Thus, while the instances it does produce suffice for certain separations, they may not be adequate for all.

Cohen forcing with locks is the following forcing notion.


\begin{definition}
Let $\Pb$ be the following notion of forcing.
\begin{enumerate}
	\item A \emph{condition} is a tuple $p = (\sigma_0, \dots, \sigma_{k-1}, f)$, where:
	\begin{itemize}
		\item $\sigma_0, \dots, \sigma_{k-1}$ are binary strings;
		\item $f$ is a function from $\{0, \dots, k-1\}$ to $\{0,1,u\}$.
	\end{itemize}
	\item A condition $q = (\tau_0, \dots, \tau_{\ell-1}, g)$ \emph{extends} $p$, written $q \leq p$, if:
	\begin{itemize}
		\item $\ell \geq k$;
		\item $f \preceq g$;
		\item $\sigma_i \preceq \tau_i$ for all $i < k$;
		\item for all $i < k$ such that $f(i) \in \{0,1\}$ and every $x$ with $|\sigma_i| \leq x < |\tau_i|$, we have $\tau_i(x) = f(i)$.
	\end{itemize}
\end{enumerate}
\end{definition}


From any generic filter $\mathcal{G}$ on $\mathbb{P}$ we can define sets $R^{\mathcal{G}}_i = \bigcup \{\sigma_i : (\sigma_0, \dots, \sigma_{k-1}, f) \in \mathcal{G} \}$ for each $i \in \omega$, where we identify binary strings with the finite sets they define, as well as $\vec{R}^{\mathcal{G}} = \seq{R^\mathcal{G}_0, R^\mathcal{G}_1, \ldots}$, which is naturally an instance of $\COH$. Similarly, we can define $f^\mathcal{G} = \bigcup \{f : (\sigma_0, \dots, \sigma_{k-1}, f) \in \Fcal \}$, which is a total function $\omega \to \{0,1,u\}$. Notice that for any $i$ such that $f^\mathcal{G}(i) = b \in \{0,1\}$, we necessarily have that $R^\mathcal{G}_i(x) = b$ for cofinitely many $x$, so we think of $R^\mathcal{G}_i$ as being ``locked'' to the value $b$ from some point on. For any $i$ such that $f^\mathcal{G}(i) = u$, there will be infinitely many $x$ such that $R^\mathcal{G}_i(x) = 0$, and infinitely many $x$ such that $R^\mathcal{G}_i(x) = 1$, so we think of $R^\mathcal{G}_i$ as being ``unlocked''.

We will exploit an atypical feature of $\vec{R}^\mathcal{G}$, namely that merely knowing which columns are unlocked allows us to compute a cohesive set for the instance.

\begin{proposition}\label{P:cohfromlocks}
	Let $\mathcal{G}$ be a sufficiently generic filter on $\mathbb{P}$. Let $U = \{i \in \omega: f^\mathcal{G}(i) = u\}$. Then $\vec{R}^\mathcal{G} \oplus U$ computes an infinite cohesive set for $\vec{R}^\mathcal{G}$.
\end{proposition}

\begin{proof}
	Clearly, each $R^\mathcal{G}_i$ for $i \in U$ is Cohen generic, and so is the join of any finite number of such $R^\mathcal{G}_i$. By genericity, $U$ is infinite, so we can list out its elements as $i_0 < i_1 < \cdots$. We now define a set $C = \{ x_0 < x_1 < \cdots\}$ computably from $\vec{R}^\mathcal{G} \oplus U$, as follows. Fix $k \in \omega$, and suppose we have defined $x_j$ for all $j < k$. Let $x_k$ be the least $x$ such that $x > x_j$ for all $j < k$, and $R_{i_j}^\mathcal{G}(x) = 1$ for all $j \leq k$. This $x$ exists because $\bigoplus_{j \leq k} R^\mathcal{G}_{i_j}$ is generic. Thus, for each $i \in U$ we have that almost all the $x_j$ belong to $R^\mathcal{G}_i$, hence $C \subseteq^* R^\mathcal{G}_i$. On the other hand, for any $i \notin U$, we have that $R^\mathcal{G}_i =^* \omega$ or $R^\mathcal{G}_i =^* \emptyset$, so trivially also $C \subseteq^* R^\mathcal{G}_i$ or $C \subseteq^* \overline{R^\mathcal{G}_i}$. Hence, $C$ is cohesive for $\vec{R}^\mathcal{G}$, as desired.
\end{proof}

\noindent Of course, the set $U$ in the above proposition is computable from $f^\mathcal{G}$, so in particular, $\vec{R}^\mathcal{G} \oplus f^\mathcal{G}$ always computes an infinite cohesive set for $\vec{R}^\mathcal{G}$.

Our main goal in this section is to prove the following result.

\begin{theorem}\label{thm:generic-cohesive-not-pazp}
	Let $\Gcal$ be a sufficiently generic filter on $\mathbb{P}$. Then $(\vec{R}^\mathcal{G} \oplus f^\mathcal{G})'$ is not of PA degree relative to ${\bf 0}'$.
\end{theorem}

\noindent Combining this theorem with the preceding proposition, we immediately get the following.

\begin{corollary}
	Let $\Gcal$ be a sufficiently generic filter on $\mathbb{P}$. Then there is an $\vec{R}^\mathcal{G}$-cohesive set $C$ such that $(\vec{R}^\mathcal{G} \oplus C)'$ is not of PA degree relative to~${\bf 0}'$. In particular, $(\vec{R}^\mathcal{G} \oplus C)'$ is not of PA degree relative to~$\deg(\vec{R}^\mathcal{G})'$, so $\vec{R}^\mathcal{G}$ is not a universal instance.
\end{corollary}

We turn to proving the theorem. In what follows, let $\vec{\mathtt{R}}$ and $\mathtt{f}$ be names in the $\mathbb{P}$ forcing language for the generic objects $\vec{R}^\mathcal{G}$ and $f^\mathcal{G}$, respectively. Let $\forces$ denote the forcing relation, as usual. We refer the reader to Shore~\cite[Chapter 3]{Shore-2016} for further background on forcing in arithmetic. The following is the main combinatorial ingredient of our proof.

\begin{lemma}\label{L:densitycombo}
	Let $\Phi$ be a $\{0,1\}$-valued Turing functional. The set of conditions $p$ such that
	\[
		p \forces (\exists x)~\Phi^{(\vec{\mathtt{R}} \oplus \mathtt{f})'}(x) \diverges~\vee~(\exists x)~\Phi^{(\vec{\mathtt{R}} \oplus \mathtt{f})'}(x) \converges = \Phi^{\emptyset'}_x(x) \converges.
	\]
	is dense in $\mathbb{P}$.
\end{lemma}

\begin{proof}
	Fix a condition $p$, and suppose there is no $p' \leq p$ forcing $(\exists x)~\Phi^{(\vec{\mathtt{R}} \oplus \mathtt{f})'}(x) \diverges$. Then for each $x$, the set of conditions $p'$ forcing that $\Phi^{(\vec{\mathtt{R}} \oplus \mathtt{f})'}(x) \converges$ is dense below $p$. Now, for each $x$ and $b \in \{0,1\}$, the sentence $\Phi^{(\vec{\mathtt{R}} \oplus \mathtt{f})'}(x) \converges = b$ is $\Sigma^0_2$ in the forcing language. As $\mathbb{P}$ is a computable notion of forcing, forcing $\Sigma^0_2$ sentences is $\Sigma^0_2$-definable (see, e.g.,~\cite[Theorem 3.2.5]{Shore-2016}). Thus, given $x$, we can find a $p' \leq p$ and a $b \in \{0,1\}$ such that $p' \forces \Phi^{(\vec{\mathtt{R}} \oplus \mathtt{f})'}(x) \converges = b$ uniformly $\emptyset'$-computably. That means that computably in $\emptyset'$, we can define an infinite sequence of conditions
	\[
		p = p_0 \geq p_1 \geq \cdots
	\]
	and an infinite sequence of bits
	\[
		b_0, b_1, \ldots \in \{0,1\}
	\]
	such that $p_x \forces \Phi^{(\vec{\mathtt{R}} \oplus \mathtt{f})'}(x) \converges = b_x$ for each $x$. But then there must be an $x$ such that $\Phi^{\emptyset'}_x(x) \converges = b_x$, as otherwise $\emptyset'$ could compute a diagonally non-computable function relative to itself. Thus, $p_x$ is the desired extension of $p$ forcing $\Phi^{(\vec{\mathtt{R}} \oplus \mathtt{f})'}(x) \converges = \Phi^{\emptyset'}_x(x) \converges$.
\end{proof}

We can now prove our theorem.

\begin{proof}[Proof of Theorem~\ref{thm:generic-cohesive-not-pazp}]
	Let $\mathcal{G}$ be a sufficiently generic filter on $\mathbb{P}$. Fix any $\{0,1\}$-valued Turing functional, and suppose $\Phi^{(\vec{R}^\mathcal{G} \oplus f^\mathcal{G})'}$ is total.	By genericity, Lemma~\ref{L:densitycombo} implies there is some $p \in \mathcal{G}$ forcing $(\exists x)~\Phi^{(\vec{\mathtt{R}} \oplus \mathtt{f})'}(x) \converges = \Phi^{\emptyset'}_x(x) \converges$. Then by genericity again, we must have that $\Phi^{(\vec{R}^\mathcal{G} \oplus f^\mathcal{G})'} \converges = \Phi^{\emptyset'}_x(x)$ for some $x$. In particular, $\Phi^{(\vec{R}^\mathcal{G} \oplus f^\mathcal{G})'}$ is not diagonally non-computable relative to $\emptyset'$, hence not of PA degree relative to $\textbf{0}'$. Since $\Phi$ was arbitrary, this completes the proof.
\end{proof}

\section{Symmetric and asymmetric constructions}\label{S:symmetric}

There exist two main techniques for constructing computability-theoretically weak solutions to $\SRT^2_2$. The first comes from the original proof by Seetapun that $\RT^2_2$ does not imply $\ACA_0$ over $\RCA_0$ (see Seetapun and Slaman~\cite[Theorem 2.1]{SS-1995}). The second is due to Cholak, Jockusch, and Slaman~\cite[Section 4]{CJS-2001}. Both have found wide application in the literature (e.g.,~\cite{DJ-2009,DP-2017,CSY-2014,Dzhafarov-2015,Dzhafarov-2016,Patey-2016b,Patey-2016,MP-TA}). And while both methods use Mathias forcing with similar conditions, the combinatorial cores of the two approaches are somewhat different, and this fact is reflected in their effectivity. For computable instances, Seetapun's proof requires a $\emptyset''$ oracle, while Cholak, Jockusch, and Slaman get away with an oracle of PA degree over ${\bf 0}'$. The latter thus has a number of advantages. For example, it can be used to show that every computable instance of $\SRT^2_2$ has a low$_2$ solution (\cite[Theorem 3.7]{CJS-2001}). There is no known proof of this fact using Seetapun's combinatorics.

But Seetapun's method seems to have some advantages of its own. Hirschfeldt, Jockusch, Kjos-Hanssen, Lempp, and Slaman~\cite[Theorem 4.5]{HJKLS-2008} introduced a version of Seetapun's argument that works below $\emptyset'$. They used this to prove the following result, which answered a question of Mileti~\cite[Question 5.3.8]{Mileti-2004}.

\begin{theorem}[{\cite[Theorem 4.5]{HJKLS-2008}}]\label{T:Delta2incomplete}
	Let $A$ and $C$ be $\Delta^0_2$ sets such that $C \nTred \emptyset$. Then $\emptyset'$ computes an infinite subset $H$ of $A$ or $\overline{A}$ such that $C \nTred H$.
\end{theorem}

\noindent In particular, every computable instance of $\D^2_2$ (and hence of $\SRT^2_2$) has an incomplete $\Delta^0_2$ solution. This is the first example of an upper bound on the strength of $\SRT^2_2$ that is not known to be provable with the Cholak, Jockusch, and Slaman technology. An essential key here is to use an \emph{asymmetric} proof: one construction that fully builds a homogeneous set of color $0$, and a separate backup construction that builds a homogeneous set of color $1$. The original construction of Seetapun, as well as the construction of Cholak, Jockusch, and Slaman, are both symmetric: they build the two homogeneous sets together, playing one off against the other.

In this section, we present one symmetric and one asymmetric construction of a homogeneous set, somewhat more directly comparing the strengths of the two. We begin with a symmetric argument in the style of Cholak, Jockusch, and Slaman.

\begin{theorem}\label{thm:2-generic-computes-incomplete-D22}
Let $A$ and $C$ be $\Delta^0_2$ sets such that $C \not \Tred \emptyset$.
If $G$ is any $2$-generic set, then $G \oplus \emptyset'$ computes an infinite subset $H$ of $A$ or $\overline{A}$ such that $C \nTred H$.
\end{theorem}

\noindent While this is only a slightly weaker version of Theorem~\ref{T:Delta2incomplete}, the underlying combinatorics of the proof are new. The main takeaway is that there are still unexplored aspects of this technique that can be exploited to (at least partially) reprove results that could previously only be obtained by more ad-hoc methods.


First, we recall the following results, which we will make use of.

\begin{lemma}[{Lawton, see~\cite[Theorem 4.1]{HJKLS-2008}}]\label{thm:low-incomplete-basis-theorem}
Let $\Ccal \subseteq 2^\omega$ be a non-empty $\Pi^0_1$ class,
and let $C_0, C_1, \dots >_T 0$ be uniformly $\Delta^0_2$.
Then $\Ccal$ has a low member $P$ such that $\forall i(C_i \not \Tred P)$.
\end{lemma}

\begin{lemma}[{Simpson~\cite[Lemma~VIII.2.9]{Simpson-2009}}]\label{thm:pa-degree-bounds-countably-coded-scott}
  Every set $X$ of PA degree computes a countably coded Scott set, i.e., a countable sequence $\seq{Z_0,Z_1,\ldots}$ such that $\{Z_0,Z_1,\ldots\}$ is an $\omega$-model of $\WKL$.
\end{lemma}

We will also need the following simple but useful fact about genericity. For completeness, we recall that a function $g : \omega \to \omega$ is \emph{hyperimmune relative to a set $Z$}, or \emph{$Z$-hyperimmune}, if for every $Z$-computable function $f$ there is an $n$ such that $f(n) < g(n)$.

\begin{lemma}\label{thm:2-generic-to-hyperimmune-functions}
Suppose that $G$ is $2$-generic. Then $G \oplus \emptyset'$ uniformly computes
a countable sequence of functions $f_0, f_1, \dots$ such that for every $i \in \omega$, the function
$f_i$ is hyperimmune relative to $\emptyset' \oplus \bigoplus_{j \neq i} f_j$.
\end{lemma}
\begin{proof}
  Write $G$ as $\bigoplus_{i \in \omega} G_i$, and for each $i$, let $f_i = p_{G_i}$, the principal function of $G_i$. By $2$-genericity,
  each $f_i = p_{G_i}$ is hyperimmune relative to $\emptyset' \oplus \bigoplus_{j \neq i} p_{G_j} = \emptyset' \oplus \bigoplus_{j \neq i} f_j$, as desired.
\end{proof}


We now come to the proof of the theorem, which is a priority construction. For partial functions $f$ and $g$, we shall write $f(x) \simeq g(x)$ to mean that $f(x) = g(x)$ if both $f(x)$ and $g(x)$ are defined. We shall also follow the convention that if a Turing functional converges on some input using a finite oracle $F$, then its use is bounded by $\max F$.

\begin{proof}[Proof of Theorem~\ref{thm:2-generic-computes-incomplete-D22}]
Fix $A$, $C$, and $G$, and let $A_0 = \overline{A}$ and $A_1 = A$.
By Lemmas~\ref{thm:low-incomplete-basis-theorem} and ~\ref{thm:pa-degree-bounds-countably-coded-scott}, there is a countable Turing ideal $\Mcal = \{ Z_0, Z_1, \dots \}$ as follows:
\begin{itemize}
	\item $\Mcal \models \wkl$;
	\item $\bigoplus_i Z_i$ is low;
	\item $C \not \Tred \bigoplus_i Z_i$.
\end{itemize}
We may assume neither $A_0$ nor $A_1$ has any infinite subset in $\mathcal{M}$, since otherwise we can take this subset and be done.

Recall that a lowness index for a low set $X$ is a natural number $e$ such that $\Phi_e^{\emptyset'} = X'$. In what follows, we write ``$X$ is (a lowness index for) a set'', etc., to mean that $X$ is a low set specified by a lowness index. In this way, we identify a given lowness index for $X$ with $X$ itself.

\bigskip
\noindent \textbf{Conditions.} A \emph{condition} is a tuple $(F_0, F_1, X)$, where $F_0 \subseteq \overline{A}$ and $F_1 \subseteq A$ are finite sets,
$X$ is (a lowness index for) an element of $\Mcal$, and $\max F_0, F_1 < \min X$. A condition $(E_0,E_1,Y)$ \emph{extends} $(F_0,F_1,X)$, written $(E_0,E_1,Y) \leq (F_0,F_1,X)$, if $F_i \subseteq E_i$ and $\min (E_i \setminus F_i) > \max F_i$, for each $i < 2$.

Note that unlike Mathias conditions, we do not demand the reservoir $X$ to be infinite, or that an extension only add new elements to the finite initial segments from the reservoir, or that $Y \subseteq X$ in the definition of extension. It may seem from this definition that we do not need the reservoirs, but their role will become apparent in the construction. We use the terms ``reservoir'' and ``condition'' here by analogy with Mathias forcing, even though our argument is not a forcing construction.

We will build an infinite $G \oplus \emptyset'$-computable sequence of conditions
\[
	(F_{0,0}, F_{1,0}, X_0) \geq (F_{0,1}, F_{1,1}, X_1) \geq \cdots \geq (F_{0,s}, F_{1,s}, X_s) \geq \cdots.
\]
We then let $H_0 = \bigcup_s F_{0,s}$ and $H_1 = \bigcup_s F_{1, s}$. We aim to satisfy the following requirements: for all $e_0, e_1 \in \omega$,
$$
\Rcal_{e_0, e_1}: \Phi_{e_0}^{H_0} \neq C~\vee~\Phi_{e_1}^{H_1} \neq C;
$$
and for all $n \in \omega$,
\[
	\Scal_n : |H_0| > n \wedge |H_1| > n.
\]
Thus, $H_0$ and $H_1$ will be infinite subsets of $A_0$ and $A_1$, respectively, and at least one $H_i$ will not compute $C$. The requirements will be satisfied by a finite injury priority argument, with a moveable marker procedure. The requirements are given the usual order.

\bigskip
\noindent \textbf{Hyperimmune functions.} Let $f_0, f_1, \dots$ be the $G \oplus \emptyset'$-computable functions given by Lemma~\ref{thm:2-generic-to-hyperimmune-functions}. Thus, each $f_i$ is hyperimmune relative to $\emptyset' \oplus \bigoplus_{j \neq i} f_j$. To each requirement $\Rcal_{e_0, e_1}$, we associate $f_{\langle e_0, e_1\rangle}$. Each function will be called finitely often in the construction of the sequence of conditions, and therefore the sequence will be $\emptyset' \oplus \bigoplus_{j \neq \langle e_0, e_1 \rangle} f_j$-computable. In particular, $f_{\langle e_0, e_1 \rangle}$ will be hyperimmune relative to this sequence. 

\bigskip
\noindent \textbf{Reverting the reservoir.} At a stage $s$, we may need to \emph{revert the reservoir}. This means that we look for the largest $t < s$ such that $X_t \cap (s, \infty) \neq \emptyset$. Such a $t$ must exist because we will have $X_0 = \omega$. We then set $(F_{0,s+1}, F_{1,s+1}, X_{s+1}) = (F_{0,s}, F_{1,s}, X_t \cap (\max F_{0,s} \cup F_{1,s},\infty))$. In this case, we also say that the reservoir was \emph{reverted to stage $t$ at stage $s$}. Note that the finite initial segments $F_{0,s}$ and $F_{1,s}$ are unchanged. 

\bigskip
\noindent \textbf{Movable marker procedure.} To each requirement $\Rcal_{e_0, e_1}$,
we shall associate a marker $m_{e_0, e_1} \in \omega$, which represents a stage at which the reservoir $X_{m_{e_0,e_1}}$ is infinite, and after which no requirement of higher priority requires attention (to be defined below). Initially, $m_{e_0, e_1} = 0$. At the end of each stage $s$ at which a requirement $\Rcal_{e_0, e_1}$ requires attention, we set $m_{e_0, e_1} = s+1$ and also $m_{e'_0, e'_1} = s +1$  for every requirement $\Rcal_{e'_0, e'_1}$ of lower priority. Moreover, if the reservoir is reverted to some stage $t$, then we set $m_{e_0, e_1} = s +1$ for every requirement $\Rcal_{e_0, e_1}$ with $m_{e_0, e_1} \geq t$. The construction will ensure that each marker eventually stabilizes to a value, and  that at any stage $s$, if $m_{e_0,e_1} \leq s$, then $X_{m_{e_0,e_1}} \supseteq X_s$. Indeed, the only time $X_s \not \supseteq X_{s+1}$ is when a reservoir is reverted or when a strategy acts.

\bigskip
\noindent \textbf{Requiring attention.} At a stage $s$, having $(F_{0,s}, F_{1,s}, X_s)$ already defined, a requirement $\Scal_n$ \emph{requires attention} if either $|F_{0,s}| \leq n$ or $|F_{1,s}| \leq n$. A requirement $\Rcal_{e_0, e_1}$ \emph{requires attention} if the following two properties hold:
\begin{itemize}
	\item[(1)] for each $i < 2$ and $x < \min X_{m_{e_0,e_1}}$, we have $\Phi_{e_i}^{F_{i,s}}(x) \simeq C(x)$;
	\item[(2)] for each $i < 2$ and $x < \min X_{m_{e_0,e_1}}$, there exists $E \subseteq X_{m_{e_0,e_1}} \cap (\max F_{i,s},\infty)$ such that $\Phi_{e_i}^{F_{i,s} \cup E}(x)\converges = C(x)$.
\end{itemize}
A requirement that does not require attention is called \emph{satisfied}.
In other words, $\Rcal_{e_0, e_1}$ requires attention if it is not already satisfied by either ensuring disagreement with $C$ (negation of property 1) or ensuring non-agreement with $C$ (negation of property 2). Note also that if we ever ensure the negation of property 1 then this will continue to hold whether the reservoir is later reverted or not. If we ever ensure the negation of property 2, then this will continue to hold if we never revert the reservoir again.

\bigskip
\noindent \textbf{Construction.}
We let $(F_{0,0}, F_{1,0}, X_0) = (\emptyset, \emptyset, \omega)$. At the beginning of a stage $s$, we assume we have already defined $(F_{0,s}, F_{1,s}, X_s)$.
Initially, we $\emptyset'$-computably check whether $X_s \cap (s, \infty) \neq \emptyset$.
If this is not the case, then we revert the reservoir, and do nothing else at this stage.
If $X_s \cap (s, \infty) \neq \emptyset$, then we pick the highest-priority requirement $\Rcal_{e_0, e_1}$ with $m_{e_0, e_1} \leq s$ or $\mathcal{S}_n$ with $n \leq s$ that requires attention at stage $s$. If there is no such requirement, we set $(F_{0,s+1}, F_{1,s+1}, X_{s+1}) = (F_{0,s}, F_{1,s}, X_s)$, and go to the next stage. If this requirement is $\mathcal{S}_n$, then we search computably in $\emptyset'$ for numbers $x_0,x_1$ such that $x_i \in X_s \cap A_i$ for each $i < 2$, or for a number $u$ such that $X_s \cap (u,\infty) = \emptyset$. The search must succeed, because if $X_s$ is infinite then its intersection with both $A_0$ and $A_1$ must be non-empty (and in fact, infinite). Otherwise, one of the $A_i$ would have $X_s$ as a subset, contrary to our assumption that $A_i$ has no infinite subset in $\Mcal$. If $x_0$ and $x_1$ are found we let $F_{i,s+1} = F_{i,s} \cup \{x_i\}$ for each $i < 2$, and let $X_{s+1} = X_s \cap (\max \{x_0,x_1\},\infty)$. If $u$ is found instead, we revert the reservoir.

Now suppose the highest priority requirement that requires attention is $\Rcal_{e_0, e_1}$. Let us write $m = m_{e_0,e_1}$ for the sake of notation. Note that since $m \leq s$, we have already defined $X_m$ in the construction. Let $\mathcal{D}$ be the $\Pi^{0,X_m}_1$ class of all $B_0 \oplus B_1$ such that $B_0 \cap B_1 = \emptyset$, $B_0 \cup B_1 = X_m$, and
$$
	(\forall i < 2)(\forall x)(\forall E_0,E_1 \subseteq B_i)~\Phi^{F_{i,s} \cup E_0}_{e_i}(x) \simeq \Phi^{F_{i,s} \cup E_1}_{e_i}(x).
$$
We consider two cases. Note that we can $\emptyset'$-computably determine which case we are in.

\bigskip
\noindent \textbf{Case 1:} $\Dcal = \emptyset$. Computably in $\emptyset'$, we search for a side $i < 2$, an $x$, and a finite set $E \subseteq A_i \cap X_m$ such that $\Phi^{F_{i,s} \cup E}_{e_i}(x)\converges = 1-C(x)$, and set $F_{i,s+1} = F_{i,s} \cup E$, set $F_{1-i, s+1} = F_{1-i,s}$, and set $X_{s+1} = X_s \cap (\max E, \infty)$. Such a set $E$ must be found since in particular, $(A_0 \cap X_m) \oplus (A_1 \cap X_m) \not \in \Dcal$. Now $\Rcal_{e_0,e_1}$ will be permanently satisfied. Indeed, we have ensured that the property (1) above can never hold again.

\bigskip
\noindent \textbf{Case 2:} $\Dcal \neq \emptyset$. Since $\Mcal \models \wkl$ and $X_m$ is low, we can $\emptyset'$-computably choose some $B_0 \oplus B_1 \in \Mcal \cap \Dcal$. In particular, $B_0 \oplus B_1$ is low, and $\emptyset'$ knows a lowness index for this set. There are two subcases.

\bigskip
\noindent \textbf{Case 2a:} \textit{$B_i \cap (f_{\langle e_0, e_1\rangle}(k), \infty) = \emptyset$ for some $i < 2$, where $k$ is the number of times $\Rcal_{e_0, e_1}$ has required attention prior to stage $s$.} Computably in $\emptyset'$, we search for an $x$ such that $\Phi^{F_{i,s} \cup E}_{e_i}(x) \simeq 1 - C(x)$ for all $E \subseteq B_{1-i}$. If $B_{1-i}$ is finite, the search will trivially succeed by our use conventions. On the other hand, if $B_{1-i}$ is infinite but no such $x$ exists, then because $B_0 \oplus B_1 \in \mathcal{D}$, it follows that $B_{1-i}$ computes $C$, contradicting that $B_{1-i} \in \Mcal$. So we may assume $x$ has been found. We then set $F_{0,s+1} = F_{0,s}$, set $F_{1,s+1} = F_{1,s}$, and set $X_{s+1} = B_{1-i} \cap (x,\infty)$. In this case, if $X_m$ is infinite, so is $B_{1-i}$ and hence also $X_{s+1}$. So $\Rcal_{e_0,e_1}$ will be permanently satisfied because property (2) will never hold again. (Of course, it may still be that $X_m$, hence $X_{s+1}$, is actually finite, in which case the latter will later be reverted.)

\bigskip
\noindent \textbf{Case 2b:} \textit{otherwise}. In this case, we assume (possibly wrongly) that both $B_0$ and $B_1$ are infinite. Computably in $\emptyset'$, we search for an $x$ such that $\Phi^{F_{0,s} \cup E}_{e_0}(x) \simeq 1 - C(x)$ for all $E \subseteq B_0$. Again, the search must succeed. Once $x$ is found, we set $F_{0,s+1} = F_{0,s}$, set $F_{1,s+1} = F_{1,s}$, and set $X_{s+1} = B_0 \cap (x,\infty)$. And again, if we were right that $B_0$ is infinite, $\Rcal_{e_0,e_1}$ will be permanently satisfied as above.

\bigskip
Each stage is concluded by updating the markers and going to the next stage.
This completes the construction.

\bigskip
\noindent \textbf{Verification.} We claim that every requirement is satisfied from some stage onwards.
Seeking a contradiction, fix the highest priority requirement that requires attention infinitely often. Clearly, this is some requirement $\Rcal_{e_0,e_1}$. Let $s$ be a stage after which no requirement of higher priority requires attention.

First, note that if $\Rcal_{e_0, e_1}$ is satisfied by Case 1 at some stage, then it never later requires attention again. This is because here we force disagreement, as witnessed by the finite initial segments of our conditions, and these grow monotonically even when the reservoirs are reverted. Thus, by our assumption, we must conclude that Case 1 never occurs.

In other words, each time $\Rcal_{e_0, e_1}$ requires attention, we end up satisfying it by Case 2. By construction, whenever we do this at some stage $s' > s$, it is by refining the reservoir to a final segment of some set $B_0$ or $B_1$ with $B_0 \cup B_1 = X_m$, where $m$ is the value of $m_{e_0,e_1}$ at stage $s'$. Our choice of $s$ implies that so long as we believe this refined reservoir to be infinite, the rest of the construction draws all further reservoirs from within it, and $\Rcal_{e_0, e_1}$ continues to be satisfied. Thus, at the first stage $s'' > s'$ at which $\Rcal_{e_0, e_1}$ requires attention again, we will have just reverted the reservoir to some stage $t_{s''} \leq m$.

Let $s_0 < s_1 < \cdots$ be all the stages $s' \geq s$ at which $\Rcal_{e_0, e_1}$ requires attention. As noted above, this means that $t_{s_0} \geq t_{s_1} \geq \cdots$. Fix $s_n$ such that $t_{s_n} = t_{s_l}$ for all $l \geq n$. We show that after stage $s_n$, Case 2a never applies again in the construction, so that $\Rcal_{e_0, e_1}$ is ever after only satisfied by Case 2b. Indeed, suppose $\Rcal_{e_0, e_1}$ requires attention at some stage $s_l \geq s_n$, and suppose we satisfy it by Case 2a. Then at this stage, we choose $B_0$ and $B_1$ as above, and refine the reservoir to a final segment of some $B_{1-i}$ because we already know that $B_i$ is finite. When we revert the reservoir right before stage $s_{l+1}$, it is because we discover that $B_{1-i}$ is also finite, meaning that in fact $B_0 \cup B_1$ is finite. But $B_0 \cup B_1$ is the entire reservoir at the start of stage $s_l$, so when we revert right before the start of stage $s_{l+1}$ it must be to a stage strictly before $s_l$. That is, $t_{s_{l+1}} < t_{s_l}$, which is a contradiction since $t_{s_{l+1}} = t_{s_l} = t_{s_n}$.

Hence, after stage $s_n$, the case analysis between Cases 2a and 2b is no longer necessary for $\Rcal_{e_0, e_1}$. Since this is the only point in the construction where we use the function $f_{\seq{e_0,e_1}}$ it follows that the construction is computable in $\emptyset' \oplus \bigoplus_{j \neq \seq{e_0,e_1}} f_j$. Now define a partial function $h : \omega \to \omega$, as follows. Let $k_0$ be the number of times $\Rcal_{e_0, e_1}$ requires attention prior to stage $s_n$, and let $h(k) = 0$ for all $k < k_0$. Now for $k \geq k_0$, let $B_0$ and $B_1$ be the sets we consider refining the reservoir to under Case 2 at stage $s_{n + k-k_0}$, and define
\[
	h(k) = (\mu p \in \omega)(\exists i < 2)~B_i \cap (p,\infty) = \emptyset.
\]
Note that this definition only requires knowing the construction, so $h$ is partial $\emptyset' \oplus \bigoplus_{j \neq \seq{e_0,e_1}} f_j$-computable.

Now if $h(k) \diverges$ for some $k \geq k_0$ then both the sets $B_0$ and $B_1$ considered at stage $s_{n + k-k_0}$ are infinite, so $B_0$, which is (up to finite difference) the reservoir we refine to under Case 2b at that stage, will never be reverted. Hence, $\Rcal_{e_0, e_1}$ will remain satisfied, which is a contradiction. It follows that $h$ is total. Since $f_{\seq{e_0,e_1}}$ is hyperimmune relative to $\emptyset' \oplus \bigoplus_{j \neq \seq{e_0,e_1}} f_j$, there must be some $k \geq k_0$ such that $f_{e_0, e_1}(k) \geq h(k)$. But then at stage $s_{n+k-k_0}$, Case 2a will correctly identify an $i < 2$ such that $B_i$ is finite, and we will satisfy $\Rcal_{e_0, e_1}$. This contradicts our assumption that Case 2a never applies again after stage $s_n$.

This contradiction completes the proof of our claim, and we conclude that all requirements are eventually permanently satisfied. This completes the verification and the proof of Theorem~\ref{thm:2-generic-computes-incomplete-D22}.
\end{proof}

Patey~\cite[Theorem 28]{Patey-2017b} proved that for every set $A$ and every hyperimmune function~$g$,
there is an infinite subset $H$ of $A$ or $\overline{A}$ such that $g$ is $H$-hyperimmune.
A natural question is whether this result can be effectivized in the case of $A$ and $f$ being $\Delta^0_2$. We adapt the asymmetric construction from~\cite{HJKLS-2008} to give an affirmative answer.


\begin{theorem}\label{Thm:newthmatthelastminute}
	Let $A$ be a $\Delta^0_2$ set and $g$ be a $\Delta^0_2$ hyperimmune function. Then $\emptyset'$ computes an infinite subset $H$ of $A$ or $\overline{A}$ such that $g$ is $H$-hyperimmune.
\end{theorem}

\noindent Note that the set $H$ above is necessarily incomplete as a $\Delta^0_2$ set, since obviously no $\Delta^0_2$ function can be hyperimmune relative to $\emptyset'$. Thus, our theorem also properly strengthens the result of~\cite[Theorem 4.5]{HJKLS-2008}.

Before proceeding to the proof, we need the following basis theorem.

\begin{lemma}\label{thm:low-hi-preserving-basis-theorem}
	Let $\Dcal \subseteq 2^\omega$ be a non-empty $\Pi^0_1$ class,
	and let $g$ be a $\Delta^0_2$ hyperimmune function.
	Then $\Dcal$ has a low member $P$ such that $g$ is $P$-hyperimmune.
\end{lemma}

\begin{proof}
	We build computably in $\emptyset'$ an infinite decreasing sequence	 of non-empty $\Pi^0_1$ classes
\[
	\Dcal = \Dcal_0 \supseteq \Dcal_1 \supseteq \cdots
\]
such that for every $s \in \omega$,
\begin{itemize}
	\item[(i)] either $(\forall P \in \Dcal_{s+1})~\Phi_s^P(s)\converges$ or $(\forall P \in \Dcal_{s+1})~\Phi_s^P(s)\diverges$;
	\item[(ii)] for some $x \in \omega$, either $(\forall P \in \Dcal_{s+1})~\Phi_s^P(x)\converges < g(x)$ or $(\forall P \in \Dcal_{s+1})~\Phi_s^P(x)\diverges$.
\end{itemize}
Suppose such a sequence exists. Since (i) forces the jump, it follows that $\bigcap_s \Dcal_s = \{P\}$ for some $P$. Since the sequence is $\Delta^0_2$, it also follows that this $P$ is low. Finally, by (ii), $g$ will be $P$-hyperimmune, as desired.

We now explain how to construct this sequence. At stage $s$, suppose we have defined $\Dcal_s$. We define $\Dcal_{s+1}$. Let $T \subseteq 2^{<\omega}$ be an infinite computable binary tree such that $[T] = \Dcal_s$. Let $T_1 \subseteq T$ be the outcome of forcing the jump on $s$, in the standard way. Thus, any $\Pi^0_1$ subclass of $[T_1]$ satisfies (i). To satisfy (ii), we search computably in $\emptyset'$ for some $x$ such that $T_2 = \{ \tau \in T_1 : \Phi_s^\tau(x)\diverges \}$ is infinite, or such that $(\exists \ell)(\forall \tau \in 2^\ell)[\tau \in T \rightarrow \Phi_s^\tau(x)\converges < g(x)]$. Such an $x$ must be found, since if the former case does not hold, the function $h : \omega \to \omega$ that on input $x$ searches for the least $\ell$ such that $(\forall \tau \in 2^\ell)~[\tau \in T \rightarrow \Phi_s^\tau(x)\converges]$ and outputs a bound on all these computations is total computable, and by hyperimmunity of $g$, we must have $h(x) < g(x)$ for some $x$. In the former case, let $\Dcal_{s+1} = [T_2]$, and in the latter case, let $\Dcal_{s+1} = [T_1]$. The class $\Dcal_{s+1}$ therefore satisfies (i) and (ii). This completes the construction and the proof.
\end{proof}

Let $\varphi(n,E,v)$ be a $\Sigma^0_1$ formula of second-order arithmetic, where $n,v$ are number variables and $E$ is a number variable coding a finite set. For an infinite set $X$, we say $\varphi$ is \emph{essential in $X$} if for every $n$ there is a sequence $E^n_0, E^n_1, \ldots \subseteq X$ such that for all $k$ we have $\max E^n_k < \min E^n_{k+1}$ and $(\exists v)~\varphi(n,E^n_k,v)$.

\begin{proof}[Proof of Theorem~\ref{Thm:newthmatthelastminute}]
	Fix $A$ and $g$. By Lemmas~\ref{thm:pa-degree-bounds-countably-coded-scott} and~\ref{thm:low-hi-preserving-basis-theorem}, there exists $\Mcal = \{ Z_0, Z_1, \dots \}$ as follows:
\begin{itemize}
	\item $\Mcal \models \wkl$;
	\item $\bigoplus_i Z_i$ is low;
	\item $g$ is $\bigoplus_i Z_i$ hyperimmune.
\end{itemize}
As usual, we may assume neither $A$ nor $\overline{A}$ has any infinite subset in $\mathcal{M}$. We consider two cases.

\bigskip
\noindent \textbf{Case 1:} \textit{there is an infinite set $X \in \Mcal$ such that for every $\Sigma^0_1$ formula $\varphi(n, F, v)$ that is essential in $X$, there is some $n \in \omega$ and some finite set $E \subseteq X \cap \overline{A}$ such that $(\exists v < g(n))~\varphi(n, E, v)$.} We build an infinite $\emptyset'$-computable subset $H$ of $\overline{A}$ such that $g$ is $H$-hyperimmune. More precisely, we build an infinite $\emptyset'$-computable sequence of finite sets $F_0, F_1, \ldots$ such that $|F_e| < |F_{e+1}|$ and $F_e \subseteq X \cap \overline{A}$ for all $e$, and such that $g$ is $\bigcup_e F_e$-hyperimmune. We set $H = \bigcup_e F_e$.

Let $F_0 = \emptyset$, and suppose by induction that we have defined $F_e \subseteq X \cap \overline{A}$ for some $e$. Let $\varphi(n,E,v)$ be the formula
\[
	\max F_e < \min E~\wedge~\Phi_e^{F_e \cup E}(n) \converges = v,
\]
which is obviously $\Sigma^0_1$. Now, either $\varphi$ is essential in $X$ or it is not. If it is not, then there must be an $n$ and an $x \geq \max F_e$ such that for all $E \subseteq X \cap (x,\infty)$ we have $\neg (\exists v)~\varphi(n,E,v)$. If $\varphi$ is essential, then by assumption there is some $n$ and some $E \subseteq X \cap \overline{A}$ such that $\max F_e < \min E$ and $\Phi_e^{F_e \cup E}(n) \converges < g(n)$. We can search for this data computably in $\emptyset'$. We consider two subcases, according to which we find first.

\bigskip
\noindent \textbf{Case 1a:} \textit{there is an $n$ and an $x \geq \max F_e$ such that for all $E \subseteq X \cap (x,\infty)$ we have $\neg (\exists v)~\varphi(n,E,v)$.} Since $A$ has no infinite subset in $\mathcal{M}$, we can find a $y > x$ in $X \cap \overline{A}$. Let $F_{e+1} = F_e \cup \{y\}$. By construction and the definition of $\varphi$, we have ensured that $\Phi_e^{H}(n) \diverges$.

\bigskip
\noindent \textbf{Case 1b:} \textit{there is some $n$ and some $E \subseteq X \cap \overline{A}$ such that $\max F_e < \min E$ and $\Phi_e^{F_e \cup E}(n) \converges < g(n)$} We then let $F_{e+1} = F_e \cup E$. Thus, we have ensured that $\Phi_e^H$ will not dominate $g$.

\bigskip
\noindent This finishes the construction. Clearly, the resulting set $H$ is infinite and computable in $\emptyset'$, and $g$ is $H$-hyperimmune, as desired.

\bigskip
\noindent \textbf{Case 2:} \textit{otherwise.} In this case, we build an infinite $\emptyset'$-computable subset $H$ of $A$ such that $g$ is $H$-hyperimmune. We construct $H$ by stages, as we now describe. In some ways, this construction is similar to (but simpler than) that of Theorem~\ref{thm:2-generic-computes-incomplete-D22}. Thus, we omit some of the details below.

\bigskip
\noindent \textbf{Conditions.} A \emph{condition} is a pair $(F,X)$, where $F \subseteq A$ is a finite set, $X$ is (a lowness index for) an element of $\Mcal$, and $\max F < \min X$. A condition $(E,Y)$ \emph{extends} $(F,X)$, written $(E,Y) \leq (F,X)$, if $F \subseteq E$ and $\min(E \setminus F) > \max F$.

We build a $\emptyset'$-computable sequence of conditions
\[
	(F_0,X_0) \geq (F_1,X_1) \geq \cdots
\]
and let $H = \bigcup_e F_e$. Our goal is to satisfy the following requirements: for all $e \in \omega$,
\[
	\Rcal_e : (\exists n)~\Phi_e^{H}(n) \diverges~\vee~(\exists n)~\Phi_e^H(n) \converges < g(n);
\]
and for all $n \in \omega$,
\[
	\Scal_n : |H| > n.
\]
Clearly, these requirements suffice for our needs. We assign the requirements priorities as usual.

\bigskip
\noindent \textbf{$\Delta^0_3$ approximation.} Since we are in Case 2, for every infinite $X \in \Mcal$ there exists a $\Sigma^0_1$ formula $\varphi$ essential in $X$ such that for every $n$ and every finite set $E \subseteq X$, if $(\exists v < g(n))~\varphi(n, E, v)$ then $E \cap A \neq \emptyset$. Now, $\emptyset''$ can find such a $\varphi$ uniformly from a lowness index for $X$. So if we fix a computable indexing $\varphi_0,\varphi_1,\ldots$ of all $\Sigma^0_1$ formulas, then there is a $\emptyset''$-computable function $f : \omega \to \omega$ such that for every (lowness index for a) low set $X$, if $X$ is infinite then $\varphi_{f(X)}$ is the $\Sigma^0_1$ formula we want, and if $X$ is finite then $f(X)$ is some arbitrary value. Let $\widehat{f}$ be a $\emptyset'$-computable approximation to $f$, so that for every $X$ we have $f(X) = \lim_s \widehat{f}(X,s)$.

\bigskip
\noindent \textbf{Movable marker procedure.} To each requirement $\Rcal_e$,
we associate a marker $m_e \in \omega$. This marker represents a stage such that $X_{m_e}$ is infinite, and after which no requirement of higher priority requires attention (as defined below). We approximate $f(X_{m_e})$ in order to satisfy $\Rcal_e$, as detailed in the construction. At the end of each stage $s$ at which $\Rcal_e$ is the highest-priority requirement that requires attention, we set $m_{e'} = s+1$ for every requirement $\Rcal_{e'}$ of strictly lower priority. Also, whenever the reservoir is reverted to some stage $m_e$ (defined below), we set $m_{e'} = s$ for every requirement $\Rcal_{e'}$ of strictly lower priority than the requirement that has caused the reservoir to be reverted. Initially, we set $m_e = 0$ for all $e$. Thus, at the start of a stage $s$, we will have $m_e \leq s$ for all $e$.

\bigskip
\noindent \textbf{Reverting the reservoir.} At a stage $s > 0$, we will sometimes need to \emph{revert the reservoir}. This means that we look for the least $e \leq s$ such that our approximation to $f(X_{m_e})$ changes at stage $s$, i.e., $\widehat{f}(X_{m_e},s) \neq \widehat{f}(X_{m_e},s-1)$.
We then say the reservoir has been \emph{reverted to stage $m_e$}, and redefine $X_s = X_{m_e} \cap (\max F_s,\infty)$. We do this, rather than defining $X_{s+1}$ to be $X_{m_e} \cap (\max F_s,\infty)$, because unlike in the construction of Theorem~\ref{thm:2-generic-computes-incomplete-D22}, we do not want resetting the reservoir to cause us to go to the next stage. This definition will allow us to reset the reservoir and then continue with other actions at stage $s$.

\bigskip
\noindent \textbf{Requiring attention.} At stage $s$, a requirement $\Scal_n$ \emph{requires attention} if $|F_s| \leq n$. A requirement $\Rcal_e$ \emph{requires attention} if the following properties hold:
\begin{enumerate}
	\item for each $x < \min X_{m_e}$, if $\Phi_e^{F_s}(x) \converges$ then $\Phi_e^{F_s}(x) \geq g(x)$;
	\item for each $x < \min X_s$, there exists $E \subseteq X_{m_e} \cap (\max F_s,\infty)$ such that $\Phi_e^{F_s \cup E}(x) \converges \geq g(x)$.
\end{enumerate}
A requirement that does not require attention is called \emph{satisfied}.

\bigskip
\noindent \textbf{Construction.} Initially, let $(F_0,X_0) = (\emptyset,\omega)$. At stage $s$, assume we have defined $(F_0,X_0),\ldots,(F_s,X_s)$. If $s > 0$ and there is an $e \leq s$ such that $\widehat{f}(X_{m_e},s) \neq \widehat{f}(X_{m_e},s-1)$, then we revert the reservoir. However, unlike in the construction in Theorem~\ref{thm:2-generic-computes-incomplete-D22}, we do not end the stage if this happens. Instead, whether this happens or not, we now consider the highest-priority requirement $\Scal_n$ or $\Rcal_e$ for $n,e \leq s$ that requires attention at stage $s$. If there is no such requirement, let $(F_{s+1},X_{s+1}) = (F_s,X_s)$. If such a requirement exists, and it is $\Scal_n$, then we search for the least $x \in X_s \cap A$, or for a number $u$ such that $\widehat{f}(X_{m_e},u) \neq \widehat{f}(X_{m_e},s)$ for some $e \leq s$. The search must succeed, because as we will see, if $\widehat{f}(X_{m_e},s) = f(X_{m_e})$ for all $e \leq s$ then $X_s$ must be infinite, and if $X_s$ is infinite then it must intersect $A$ by our assumption that $\overline{A}$ has no infinite subset in $\Mcal$. Now if $x$ is found, we let $F_{s+1} = F_s \cup \{x\}$ and $X_{s+1} = X_s \cap (x,\infty)$. If $u$ is found instead, we revert the reservoir, and we start our search for the highest-priority requirement requiring attention again.

Now suppose the highest-priority requirement requiring attention is $\Rcal_e$. Write $\varphi = \varphi_{\widehat{f}(X_{m_e},s)}$ for ease of notation. Computably in $X_{m_e}$, we build for each $n$ a sequence $E^n_0, E^n_1, \ldots \subseteq X_{m_e}$ such that $\max F_s < \min E^n_k$ and $\max E^n_k < \min E^n_{k+1}$, and $(\exists v)~\varphi(n,E^n_k,v)$ for all $k$. Note that if $X_{m_e}$ is infinite and $\widehat{f}(X_{m_e},s) = f(X_{m_e})$ then each of these sequences is actually infinite. Otherwise, we may not be able to find $E^n_k$ for some $n$ and $k$, so some of the sequences may be finite (partial). For each $n$, let $T_n$ be the set of all $\alpha \in \omega^{<\omega}$ such that $\alpha(k) \in E^n_k$ for all $k < |\alpha|$. Thus, if $E^n_0, E^n_1, \ldots$ is infinite, then $T_n$ is an infinite $X_{m_e}$-computable, $X_{m_e}$-computably bounded tree. Otherwise, $T_n$ may be only partially defined. But either way, we can uniformly find a $\Delta^{0,X_{m_e}}_1$ index for $T_n$ as a (possibly partial) tree. Using this index and the lowness index of $X_{m_e}$, we can ask $\emptyset'$ whether
\[
	U_n = \{ \alpha \in T_n : (\forall E \subseteq \ran(\alpha))~\Phi_e^{F_s \cup E}(n) \diverges \}
\]
is finite, by which we mean that there is a level $k$ at which $T_n$ is defined, but such that $U_n$ has no elements of length $k$. If the answer is yes, then $U_n$ is actually a finite tree, regardless of whether the sequence $E^n_0, E^n_1, \ldots$ is infinite or not. If the answer is no, then $U_n$ is a bona fide infinite tree, provided $E^n_0, E^n_1, \ldots$ is infinite as well. Of course, if the answer is no but $E^n_0, E^n_1, \ldots$ is finite, so that $T_n$ is only partially defined, then we will be incorrectly assuming that $U_n$ is infinite. But if this is so, we will eventually revert the reservoir.

Define a partial $X_{m_e}$-computable function $h : \omega \to \omega$ as follows. On input $n$, the function $h$ first searches for the least $\ell$ such that $E^n_k$ is defined for all $k < \ell$ and
\[
	(\forall \alpha \in T_n)~[|\alpha| = \ell \implies (\exists E \subseteq \ran(\alpha))~\Phi^{F_s \cup E}_e(n) \converges],
\]
and then outputs the least $w$ bounding the values of all the relevant computations $\Phi^{F_s \cup E}_e(n)$ as well as all the witnesses $v$ such that $\varphi(n,E^n_k,v)$ for $k < \ell$.

We now $\emptyset'$-computably search for the least $n$ such that either $h(n) \converges < g(n)$ or such that $U_n$ is not finite in the sense described above. As mentioned above, if $\emptyset'$ thinks that $U_n$ is finite then it is actually so, so $h(n)$ will be defined. Hence, if we never find an $n$ such that $U_n$ is not finite then $h$ will be a total $X_{m_e}$-computable function, and there will have to be an $n$ such that $h(n) < g(n)$ since $g$ is $X_{m_e}$-hyperimmune. It follows that our search must succeed. So fix $n$; we have two subcases.

\bigskip
\noindent \textbf{Case 2a:} \textit{$U_n$ is not finite.} Then $\emptyset'$ can produce (a lowness index for) a low path $X \in [U_n] \cap \mathcal{M}$. Note that if $T_n$ is really infinite (i.e., not partially defined) then the range of every path of $U_n$ is infinite. So it makes sense to identify $X$ with its range. Also, since each $E^n_k$ was a subset of $X_{m_e}$, so is $X$. We let $(F_{s+1},X_{s+1}) = (F_s,X)$. Now if the reservoir is never again reverted, we will have satisfied $\Rcal_e$ by ensuring that property (2) in the definition of requiring attention never holds again.

\bigskip
\noindent \textbf{Case 2b:} \textit{$h(n) \converges < g(n)$.} Let $\ell$ be the level of $T_n$ witnessing that $h(n) \converges$. Then for all $k < \ell$ we have that $(\exists v < h(n))~\varphi(n,E^n_k,v)$, and hence that $(\exists v < g(n))~\varphi(n,E^n_k,v)$. Since we are in Case 2, this means that for every $k < \ell$ there is some $x_k \in E^n_k \cap A$. Then $\alpha = x_0 \cdots x_{\ell-1}$ is an element of $T_n$ of length $\ell$, hence there exists some $E \subseteq \ran(\alpha) \subseteq A$ such that $\Phi_e^{F_s \cup E}(n) \converges$, and by definition, we have $\Phi_e^{F_s \cup E}(n) \leq h(n) < g(n)$. In this case, we set $F_{s+1} = F_s \cup E$ and $X_{s+1} = X_{m_e} \cap (\max E,\infty)$. We have now permanently satisfied $\Rcal_e$, because property (1) in the definition of requiring attention will never hold again.

\bigskip
Each stage is concluded by going to the next stage. This completes the construction.

\bigskip
\noindent \textbf{Verification.} This verification is similar to (but simpler than) that of Theorem~\ref{thm:2-generic-computes-incomplete-D22}. Seeking a contradiction, fix the highest-priority requirement that requires attention infinitely often. Let $s$ be a stage after which no requirement of higher priority requires attention again. First, note that this requirement cannot be $\Scal_n$. Otherwise, at the first stage $s' \geq s$ at which $\Scal_n$ requires attention, we would by construction have to end up resetting the reservoir infinitely many times. But that is impossible, because eventually our approximations to $f(X_{m_e})$ for all $e \leq s'$ are correct, and $\Scal_n$ is then (permanently) satisfied. So, the requirement in question must be some $\Rcal_e$. Without loss of generality, assume $\widehat{f}(X_{m_{e'}},s') = f(X_{m_{e'}})$ for all $s' \geq s$ and all $e'$ such that $\Rcal_{e'}$ has higher priority than $\Rcal_e$ or is equal to $\Rcal_e$. By induction on all such $e'$, it follows that $m_{e'}$ never changes again after stage $s$ and that $X_{m_{e'}}$ is infinite. Now consider any stage $s' \geq s$ at which $\Rcal_e$ requires attention (after any resets of the reservoir). Since $X_{m_e}$ is infinite, all the sequences $E^n_0,E^n_1,\ldots$ that we define at this stage will be infinite, and all the trees $U_n$ will thus either be actually finite or actually infinite. This means that whether Case 2a applies or Case 2b applies, we will end up permanently satisfying $\Rcal_e$ at this stage, which is a contradiction.

Thus, all requirements are eventually permanently satisfied, so $H = \bigcup_s F_s$ is an infinite $\emptyset'$-computable subset of $A$ and $g$ is $H$-hyperimmune.
\end{proof}


The original asymmetric proof of Theorem \ref{T:Delta2incomplete} in \cite{HJKLS-2008} breaks into two cases, depending on whether the set $A$ is or is not hyperimmune. In the former case, the construction actually produces an infinite subset of $\overline{A}$ that is low (see \cite[Corollary 4.9]{HJKLS-2008}), while in the latter, it produces an infinite subset of $A$ that is merely incomplete $\Delta^0_2$. By Downey, Hirschfeldt, Lempp, and Solomon~\cite{DHLS-2001}, we cannot improve this proof to obtain a low set in either case. In particular, there is no hope of proving that every non-hyperimmune $\Delta^0_2$ set $A$ has an infinite low subset. However, as we show next, we can obtain this conclusion if we work with a variation on the notion of hyperimmunity.

\begin{definition}
Let $\Mcal$ be a Turing ideal.
\begin{enumerate}
	\item Let $\varphi(D)$ be a formula of second-order arithmetic, where $D$ is a number variable coding a finite set. For an infinite set $X \in \Mcal$, we say $\varphi$ is \emph{$\Mcal$-densely essential within $X$} if for every infinite $Y \subseteq X$ in $\Mcal$ there is a non-empty finite set $D \subseteq Y$ such that $\varphi(D)$ holds.
	\item A set $A$ is \emph{densely $\Mcal$-hyperimmune} if for every infinite set $X \in \Mcal$ and every $\Sigma^0_1(X)$ formula $\varphi$ that is $\Mcal$-densely essential within $X$, there is a finite set $D \subseteq X \cap \overline{A}$ such that $\varphi(D)$ holds.
\end{enumerate}
\end{definition}

\begin{theorem}\label{thm:incomplete-delta2-not-densely-hyperimmune}
Fix a Scott ideal $\Mcal$ coded by a low set and a $\Delta^0_2$ set $A$ that is not densely $\Mcal$-hyperimmune. Then there is a low infinite set $G \subseteq A$.
\end{theorem}
\begin{proof}
Let $X \in \Mcal$ be an infinite set and $\varphi(D)$ be a $\Sigma^0_1(X)$ formula
witnessing that $A$ is not densely $\Mcal$-hyperimmune. Thus $\varphi$ is $\Mcal$-densely essential within $X$, and for every set $D \subseteq X$ such that $\varphi(D)$ holds, we have $D \cap A \neq \emptyset$. We build a $\Delta^0_2$ decreasing sequence of Mathias conditions
\[
	(\emptyset, X) = (F_0, X_0) \geq (F_1, X_1) \geq \cdots
\]
such that $F_e \subseteq A$ and $X_e \in \Mcal$ for all $e$. We then take $G = \bigcup_e F_e$.

The sequence is defined inductively. Suppose we have already defined $(F_e, X_e)$. Let $D_0, D_1, \ldots \subseteq X_e$ be an infinite $X \oplus X_e$-computable sequence of non-empty finite sets such that $\max D_n < \min D_{n+1}$ and $\varphi(D_n)$ holds for each $n \in \omega$. Such a sequence exists since $\varphi$ is $\Mcal$-densely essential within $X$.  Let $T$ be the $X \oplus X_e$-computable tree of all strings $\alpha \in \omega^{<\omega}$ such that $\alpha(n) \in D_n$ for each $n < |\sigma|$. Thus, $T$ is an infinite, $X \oplus X_e$-computable, $X \oplus X_e$-computably bounded tree, and (the range of) every path through $T$ is an infinite set. Moreover, our assumption on $\varphi$ implies that there is such a path that is a subset of $A$. Now, define
\[
	U = \{ \alpha \in T : (\forall E \subseteq \ran(\alpha))~\Phi_e^{F_e \cup E}(e) \diverges \}.
\]
We have two cases:

\medskip
\noindent \textbf{Case 1:} \emph{$U$ is finite.} Since $T$ has a path that is an infinite subset of $A$ while $U$ has no paths, we can fix $\alpha \in T \setminus U$ such that $\ran(\alpha) \subseteq A$. Then we can choose a finite set $E \subseteq \ran(\alpha)$
such that $\Phi_e^{F_e \cup E}(e) \converges$. Since $A$ is not densely $\mathcal{M}$-hyperimmune and $X_e \subseteq X$, we can find $x > \max E$ in $A \cap X_e$. Define $F_{e+1} = F_e \cup E \cup \{x\}$ and $X_{e+1} = X_e \cap (x,\infty)$. The condition $(F_{e+1}, X_{e+1})$ now forces $e \in G'$.

\medskip
\noindent \textbf{Case 2:} \emph{$U$ is infinite.} Since $U \in \mathcal{M}$, we can uniformly $\emptyset'$-computably find (a lowness index for) a path $Y$ through $U$ in $\mathcal{M}$. Since $A$ is not densely $\mathcal{M}$-hyperimmune and $Y \subseteq X$, we can find $x \in A \cap Y$. Define $F_{e+1} = F_e \cup \{x\}$ and $X_{e+1} = Y \cap (x,\infty)$. Then the condition $(F_{e+1}, X_{e+1})$ forces $e \notin G'$.

\medskip
\noindent This completes the construction of our sequence of conditions, which is clearly a $\Delta^0_2$ sequence. As the case distinction above is uniform in $\emptyset'$, it follows that $G$ is low. And since we add at least one new element to $G$ at each stage, $G$ is infinite. This completes the proof.
\end{proof}

The following immediate corollary points, in some sense, to the narrowness of the class of examples of $\Delta^0_2$ sets having no low infinite subsets in them or their complements.

\begin{corollary}
	Let $A$ be a $\Delta^0_2$ set with no low infinite subset in it or its complement. Then neither $A$ nor $\overline{A}$ is hyperimmune, but each is $\mathcal{M}$-densely hyperimmune, for every Scott ideal $\mathcal{M}$ coded by a low set.	
\end{corollary}

In conclusion, we note that we do not know if Theorem \ref{thm:incomplete-delta2-not-densely-hyperimmune} could be used as part of a new proof of Theorem \ref{T:Delta2incomplete}, i.e., if there is such a proof where the case distinction could be based on whether or not $A$ is densely $\Mcal$-hyperimmune, rather than just plain hyperimmune, as in the original proof. More specifically, we do not know the answer to the following question:

\begin{question}
	Fix a Scott ideal $\Mcal$ coded by a low set and a $\Delta^0_2$ set $A$ that is densely $\Mcal$-hyperimmune. Must $\overline{A}$ have an incomplete $\Delta^0_2$ infinite subset?
\end{question}

\section{Cohesiveness and variants of hyperimmunity}\label{S:hyperimmunity}


In this section, we study variations of hyperimmunity notions to broaden the class of computable instances of $\srt^2_2$ that are known to have solutions that do not compute a solution to every computable instance of $\COH$. As mentioned above, by Jockusch and Stephan's result~\cite[Theorem 2.1]{JS-1993}, these are precisely the computable instances of $\SRT^2_2$ having solutions $H$ satisfying $\deg(H)' \not\gg {\bf 0}'$.

For the purposes of the definition below, we say a collection $\mathcal{C}$ of sets is \emph{downward closed} if it is downward closed under inclusion. Also, we use \emph{array} to mean a sequence of canonical indices of finite sets $D_0,D_1,\ldots$ such that $\lim_n \min D_n = \infty$.

\begin{definition}
Fix $X,Z \subseteq \omega$. A downward closed collection $\Ccal$ of finite sets is \emph{$Z$-hyperimmune within $X$} if for every $Z$-computable array $D_0, D_1, \ldots \subseteq X$, there is some $n \in \omega$ such that $D_n \in \Ccal$. 
\end{definition}

Whenever $X = \omega$, we simply say that $\Ccal$ is $Z$-hyperimmune.
This general notion can be used to define many notions of hyperimmunity. 
For example, we can say that a set $A$ is \emph{$Z$-hyperimmune within $X \subseteq \omega$} if $\{ F : F \subseteq \overline{A} \}$ is $Z$-hyperimmune within $X$. When $X = \omega$, this agrees with the usual definition of $A$ being $Z$-hyperimmune (see \cite[Definition 5.3.1~(iii)]{Soare-2016}).

Our starting point is the following theorem, which is a variation on the aforementioned result of Hirschfeldt, Jockusch, Kjos-Hanssen, Lempp, and Slaman~\cite[Corollary 4.9]{HJKLS-2008} that the complement of any $\Delta^0_2$ hyperimmune set $A$ has an infinite subset of low degree.

\begin{theorem}\label{thm:delta2-joint-hyperimmunity-low}
Fix $X \subseteq \omega$. Let $\Ccal_0$ and $\Ccal_1$ be $\Delta^{0,X}_2$ downward closed collections of finite sets such that $\Ccal_0 \cup \Ccal_1$ is hyperimmune within $X$, and let $f$ be a $\Delta^{0,X}_2$ function from the set of (canonical indices of) all finite sets to $\omega$. There is an $i < 2$ and a sequence of non-empty finite sets $F_0,F_1,\ldots \in \mathcal{C}_i$ as follows:
\begin{itemize}
	\item $F_s \subseteq X$ for all $s$;
	\item $\max F_s < \min F_{s+1}$ for all $s$;
	\item $f(F_s) <  \min F_{s+1}$ for all $s$;
	\item there is an infinite set $G \subseteq \bigcup_s F_s$ that is low over $X$.
\end{itemize}
\end{theorem}
\begin{proof}
We prove the result for $X = \omega$. The general case follows by a straightforward relativization of our proof. Uniformly in $\emptyset'$, we build a sequence of pairs of binary strings
\[
	(\sigma_{0,0}, \sigma_{1,0}), (\sigma_{0,1}, \sigma_{1,1}), \dots.
\]
For each $i$, let $E_{i,0} = \emptyset$, and for each $s$ let
\[
	E_{i,s+1} = \{|\sigma_{i,s}| \leq x < |\sigma_{i,s+1}| : \sigma_{i,s+1}(x) = 1\}.
\]
Also, let $G_i = \bigcup_s E_{i,s}$.

We will ensure that for each $i$ and $s$ the following hold:
\begin{itemize}
	\item $\sigma_{i,s} \prec \sigma_{i,s+1}$;
	\item $E_{i,s} \in \Ccal_i$;
	\item $f(E_{i,s}) < \min E_{i,t}$ for all $t > s$ for which $E_{i,t}$ is non-empty.
\end{itemize}
We let $\seq{\sigma_{0,0},\sigma_{1,0}} = \seq{\emptyset,\emptyset}$, and then proceed by stages. We define $\sigma_{i,s+1}$ at stage $s$, and ensure that at stage $s=\seq{e_0,e_1}$ there is an $i < 2$ such that
\[
	{\Phi^{\sigma_{i,s+1}}_{e_i}(e_i)\converges} \vee {(\forall \rho \in 2^{\omega})~\Phi^{\sigma_{i,s+1}\rho}_{e_i}(e_i) \diverges}.
\]
It follows that there is an $i < 2$ such that for each $e$ there is an $s$ so that
\[
	{\Phi^{\sigma_{i,s}}_{e}(e)\converges} \vee {(\forall \rho \in 2^{\omega})~\Phi^{\sigma_{i,s}\rho}_{e}(e) \diverges}.
\]
Hence, $G_i$ is low. Moreover, $G_i$ must be infinite. To see this, suppose not, and let $k = \max G_i$. Consider an $e \in \omega$ such that for all oracles $X$ and inputs $x$ we have that $\Phi^X_e(x) \converges$ if and only if $X \cap (k, \infty) \neq \emptyset$. Then for all $s$ we have that $\Phi^{\sigma_{i,s}}_e(e) \diverges$, yet there is always a $\rho \in 2^{<\omega}$ such that $\Phi^{\sigma_{i,s} \rho}_e(e) \converges$. This is a contradiction. Now since $G \subseteq \bigcup_s E_{i,s}$, it follows that we can computably pick out those $E_{i,s}$ that are non-empty, renaming the new sequence $F_0,F_1,\dots$. Taking this sequence together with $G = G_i$ yields the theorem.

We have thus only to construct the $\sigma_{i,s}$. At stage $s =\langle e_0,e_1\rangle$, assume inductively that we have already defined $(\sigma_{0,s},\sigma_{1,s})$. For each $i$, let $\tau_i = \sigma_{i,s}0^{f(E_{i,s})+1}$, so that $\sigma_{i,s} \prec \tau_i$. Now, computably in $\emptyset'$, we search for an $i < 2$ such that one of the following holds:
\begin{enumerate}
	\item there is some finite string $\rho \in 2^{<\omega}$ such that $\{|\tau_i| + x : \rho(x) = 1 \} \in \Ccal_i$ and $\Phi^{\tau_i\rho}_{e_i}(e_i)\converges$;
	\item there is some $n \in \omega$ such that $\Phi^{\tau_i0^n\rho}_{e_i}(e_i)\diverges$ for all $\rho \in 2^{<\omega}$.
\end{enumerate}
In the first case, we let $\sigma_{i,s+1} = \tau_i\rho$, and in the second case, we let $\sigma_{i,s+1} = \tau_i0^n$. We let $\sigma_{1-i,s+1} = \sigma_{1-i,s}0$. Clearly, these extensions are of the desired sort. Thus, the only thing left is to show that the search above must succeed. Indeed, if (2) fails for each $i < 2$, then for each $n$ we can computably find strings $\rho_0,\rho_1$ such that $\Phi^{\tau_i0^n\rho_i}_{e_i}(e_i)\converges$ for each $i < 2$. Let
\[
	D_{i,n} = \{ |\tau_i| + n \leq x <  |\tau_i| + n + |\rho_i| : \rho_i(x - |\tau_i| - n) = 1 \},
\]
and let $D_n = D_{0,n} \cup D_{1,n}$. This defines a computable array $D_0, D_1, \ldots$, so by hyperimmunity of $\mathcal{C}$, we must have $D_n \in \mathcal{C}$ for some $n$. Fix $i < 2$ so that $D_n \in \mathcal{C}_i$. By downward closure, we also have $D_{i,n} \in \mathcal{C}_i$. But then $\rho = 0^n\rho_i$ witnesses that (1) above holds for $i$, which proves the claim. This completes the proof of Theorem~\ref{thm:delta2-joint-hyperimmunity-low}.
\end{proof}


The following special case of the theorem is perhaps the more noteworthy result here, though we shall make use of the full technical version in our proof of Proposition \ref{P:DHYT} below.

\begin{corollary}
	Let $\Ccal_0$ and $\Ccal_1$ be $\Delta^0_2$ downward closed collections of finite sets such that $\Ccal_0 \cup \Ccal_1$ is hyperimmune. Then there is a low infinite set $H$ and an $i < 2$ such that $H = \bigcup_s D_s$ for some $D_0,D_1,\ldots \in \Ccal_i$.
\end{corollary}

\noindent Of particular interest is the case when $\Ccal_0 = \{F \subseteq \omega : F \text{ is finite} \wedge F \subseteq A \}$ and $\Ccal_1 = \{F \subseteq \omega : F \text{ is finite} \wedge F \subseteq \overline{A} \}$.

\begin{corollary}
For every $\Delta^0_2$ set $A$ such that the collection of finite subsets of $A$ and $\overline{A}$ is hyperimmune, there is a low infinite subset $H$ of $A$ or $\overline{A}$.
\end{corollary}


In the context of the $\SRT^2_2$ vs.\ $\COH$ problem, the fact that the complement of a hyperimmune $\Delta^0_2$ set always has an infinite low subset stands out next to the fact that no low set can compute a solution to every computable instance of $\COH$. This motivates the following definition, and makes the subsequent result surprising. 

\begin{principle}[$\Delta^0_2$ hyperimmunity ($\dhyp_k$)]
	For every set $Z$ and every $\Delta^{0,Z}_2$ $k$-partition $A_0 \sqcup \dots \sqcup A_{k-1} = \omega$, there is some $i < k$ and an infinite set $X$ such that $\overline{A}_i$ is $Z \oplus X$-hyperimmune within~$X$.
\end{principle}

%
%
%
%
%
%

\begin{proposition}\label{lem:coh-vs-srt22-dhyp2}
$\coh \cred \srt^2_2$ if and only if $\coh \cred \dhyp_2$.
\end{proposition}

\begin{proof}
First, note that $\dhyp_2 \cred \srt^2_2$, since every solution to any instance of $\SRT^2_2$ is also a solution to it viewed as an instance of $\dhyp_2$. Thus, if $\coh \cred \dhyp_2$, then $\coh \cred \srt^2_2$. For the other direction, suppose that $\coh \not\cred \dhyp_2$, and fix an instance $\vec{R} = R_0, R_1, \dots$ of $\COH$ witnessing the fact. By adding all the primitive $\vec{R}$-recursive sets to $\vec{R}$ if necessary, we can assume that every infinite $\vec{R}$-cohesive set $C$ satisfies $\deg (\vec{R} \oplus C)' \gg \deg(\vec{R})'$. We claim that  $\vec{R}$ also witnesses that $\coh \not\cred \srt^2_2$. Indeed, consider any $\vec{R}$-computable instance $c : [\omega]^2 \to 2$ of $\SRT^2_2$, and let $A_i = \{ x \in \omega: \lim_y c(x,y) = i\}$ for each $i < 2$. By choice of $\vec{R}$, there is an infinite set $X$ and an $i < 2$ such that $A_i$ is $\vec{R} \oplus X$-hyperimmune within $X$, but $\vec{R} \oplus X$ does not compute any infinite $\vec{R}$-cohesive set. Thus, $\deg(\vec{R} \oplus X)' \not\gg \deg(\vec{R})'$.

Hirschfeldt, Jockusch, Kjos-Hanssen, Lempp, and Slaman~\cite[Corollary 4.8]{HJKLS-2008} showed that for sets $U$ and $V$, if $U \subseteq V$ and $U$ is $V$-hyperimmune, then there is a set $W$ with the following properties:
\begin{itemize}
	\item $U \subseteq W \subseteq V$;
	\item $W$ is low over $V$;
	\item $V \setminus W$ is infinite.
\end{itemize}
Now, the fact that $A_i$ is $\vec{R} \oplus X$-hyperimmune within $X$ implies that $A_i \cap X$ is $\vec{R} \oplus X$-hyperimmune. (Indeed, given any $\vec{R} \oplus X$-computable array each of whose terms intersects $A_i$ we could form a new $\vec{R} \oplus X$-computable array by intersecting each term of the original with $X$. Since $A_i \subseteq X$, this new array would then witness that $A_i$ is not $\vec{R} \oplus X$-hyperimmune within $X$, a contradiction.) Taking $U = A_i \cap X$ and $V = X$, we can relativize the above result to $\vec{R}$ to find a set $W$ such that $A_i \cap X \subseteq W \subseteq X$ and such that $W$ is low over $\vec{R} \oplus X$ and $X \setminus W$ is infinite. Let $Y = X \setminus W$. Then $Y$ is still low over $\vec{R} \oplus X$, and since $Y = X \cap \overline{W} \subseteq \overline{A_i \cap X} = A_{1-i} \cup \overline{X}$, it follows that $Y$ is an infinite subset of $A_{1-i}$. To conclude, we can $c$-computably thin out $Y$ to obtain an infinite homogeneous set $H \subseteq Y$ for $c$. Then $H \Tred \vec{R} \oplus X \oplus Y$, and as such, is low over $\vec{R} \oplus X$. In particular, $(\vec{R} \oplus H)' \Tred (\vec{R} \oplus X)'$, so $\deg(\vec{R} \oplus H)' \not\gg \deg(\vec{R})'$. Thus, $\vec{R} \oplus H$ does not compute an infinite $\vec{R}$-cohesive set.
\end{proof}

Since $\dhyp_2 \cred \SRT^2_2$, one might expect the question of whether $\coh \cred \dhyp_2$ to be more combinatorially accessible than that of whether $\coh \cred \srt^2_2$. We can extend this situation a bit further. The following definition is essentially due to Wang \cite[Section 3.4]{Wang-2016}.

\begin{definition}
	Let $c : [\omega]^2 \to 2$	be a stable coloring. Say a set $F \subseteq \omega$ is \emph{$c$-compatible} if for all $x < y$ in $F$,
	\begin{itemize}
		\item if $c(x,y) = 0$ then $\lim_z c(x,z) \leq \lim_z c(y,z)$,
		\item if $c(x,y) = 1$ then $\lim_z c(x,z) \geq \lim_z c(y,z)$.
	\end{itemize}
\end{definition}

\begin{definition}
Fix $Z \subseteq \omega$. A stable coloring $c : [\omega]^2 \to 2$ is \emph{$Z$-hypertransitive within $X \subseteq \omega$} if $\{ F \subseteq \omega : F \mbox{ finite and } c\mbox{-compatible}\}$ is $Z$-hyperimmune within~$X$.
\end{definition}

\begin{principle}[$\Delta^0_2$ hypertransitivity ($\dhyt$)]
	For every set $Z$ and every $\Delta^{0,Z}_2$ stable coloring $c : [\omega]^2 \to 2$, there is an infinite set $X$ such that $c$ is $Z \oplus X$-hypertransitive within~$X$.
\end{principle}

Notice that if $\lim_y c(x,y)$ is the same for all $x$ in some $F$, then $F$ is compatible for $c$. From this fact, it follows at once that $\dhyt \cred \dhyp_2$. Furthermore, if $F_0$ and $F_1$ are non-empty, $c$-compatible, finite sets with $\max F_0 < \min F_1$, and if for each $x \in F_0$ the color $c(x,y)$ has stabilized by $\min F_1$, then $F_0 \cup F_1$ is compatible for $c$ as well. We shall make use of this observation in the proof below. 

We will need one additional fact. Recall that a set $T$ is \emph{transitive} for a coloring $c$ if for all $x < y < z$ in $T$ we have that $c(x,y) = c(y,z) \implies c(x,y) = c(x,z)$. In computability, this notion was first studied by Hirschfeldt and Shore \cite[Section 5]{HS-2007}. For us, an important fact is that if $c$ is stable and $S$ is an infinite $c$-compatible set then there exists an infinite $c \oplus S$-computable transitive set $T$ for $c$ contained in $S$. To show this, we build $T$ inductively. Let $x_0$ be the least element of $S$, and assume that for some $n \in \omega$ we have already chosen $x_0 < \cdots < x_n$ in $S$, and that these form a transitive set. Let $x_{n+1}$ be the least $x > x_n$ in $S$ such that $x_0 < \cdots < x_n < x$ forms a transitive set. Such an $x$ must exist. Indeed, fix $s$ so that for each $j \leq n$ the color of $c(x_j,y)$ stabilizes by $s$. Now for any $j < k \leq n$, any $i < 2$, and any $x > s$, if $c(x_j,x_k) = c(x_k,x) = i$ then $\lim_y c(x_k,y) = i$ by choice of $x$ and $s$, so $\lim_y c(x_j,y) = i$ by $c$-compatibility, so $c(x_j,x) = i$.

In the proof below, we invoke the principle $\semo$, first introduced by Lerman, Solomon, and Towsner \cite[Section 1]{LST-2013}, which states that every stable coloring $c : [\omega]^2 \to 2$ has an infinite transitive set. Patey \cite[Corollary 3.7]{Patey-TA4} showed that $\COH \cred \SRT^2_2$ if and only if $\COH \cred \semo$.



\begin{proposition}
\label{P:DHYT}
$\coh \cred \srt^2_2$ if and only if $\coh \cred \dhyt$.
\end{proposition}
\begin{proof}
Since $\dhyt \cred \dhyp_2$, if $\coh \cred \dhyt$ then by Proposition~\ref{lem:coh-vs-srt22-dhyp2}, $\coh \cred \srt^2_2$. In the other direction, suppose that $\coh \not\cred \dhyt$, as witnessed by the $\COH$ instance $\vec{R} = R_0, R_1, \dots$. Without loss of generality, we can assume that every infinite $\vec{R}$-cohesive set $C$ satisfies $\deg (\vec{R} \oplus C)' \gg \deg(\vec{R})'$. We claim that  $\vec{R}$ also witnesses that $\coh \not\cred \semo$, and hence that $\COH \ncred \SRT^2_2$. Let $c : [\omega]^2 \to 2$ be any $\vec{R}$-computable instance of $\semo$, i.e., a stable coloring. By choice of $\vec{R}$, there is an infinite set $X$ such that $c$ is $\vec{R} \oplus X$-hypertransitive within $X$, but $\vec{R} \oplus X$ does not compute any infinite $\vec{R}$-cohesive set. Let $\Ccal_0 = \emptyset$ and $\Ccal_1 = \{ F : F \mbox{ is } c\mbox{-compatible}\}$. Thus $\Ccal_0$ and $\Ccal_1$ are $\Delta^0_2$ downward closed collections of finite sets whose union is $\vec{R} \oplus X$-hyperimmune within $X$. Let $f$ be the $\Delta^0_2$ function that, on input of a finite set $F$, outputs the least $s$ such that for each $x \in F$ the color of $c(x,y)$ stabilizes by $s$. We can then relativize Theorem~\ref{thm:delta2-joint-hyperimmunity-low} to $\vec{R}$ and apply it to $\mathcal{C}_0$, $\mathcal{C}_1$ and $f$ to obtain a sequence of sets $F_0,F_1,\ldots \subseteq X$ and an infinite set $G \subseteq \bigcup_s F_s$ such that all the $F_s$ are $c$-compatible, $f(F_s) < \min F_{s+1}$ for all $s$, and $G$ is low over $\vec{R} \oplus X$. By the remark above, it follows that $\bigcup_s F_s$ is $c$-compatible, hence $G$ is as well. As also noted above, this means $G$ contains a $c \oplus G$-computable infinite transitive set $T$ for $c$. So $T$ is also low over $\vec{R} \oplus X$, and hence $\vec{R} \oplus T$ cannot compute any infinite $\vec{R}$-cohesive set.
\end{proof}

The previous techniques rely on the $\Delta^0_2$ approximations of the instance of $\D^2_2$. In the general setting, the following question remains open:

\begin{question}
Is there a hyperimmune set $A$ such that every infinite subset $H \subseteq \overline{A}$ satisfies $\deg(H)' \gg {\bf 0}'$?
\end{question}

For completeness, we mention also that it would be good to figure out the precise relationships between the principles $\dhyt$, $\dhyp_2$, $\semo$, and $\srt^2_2$. It is not difficult to see that $\dhyt \cred \semo$. By results of Lerman, Solomon, and Towsner \cite[Theorem 1.15]{LST-2013} we know that $\srt^2_2 \ncred \semo$, and so $\srt^2_2 \ncred \dhyt$. However, the other reductions remain open.

\begin{question}\
	\begin{enumerate}
		\item Is it the case that $\semo \cred \dhyt$?
		\item Is it the case that $\dhyp_2 \cred \dhyt$?
		\item Is it the case that $\srt^2_2 \cred \dhyp_2$?
	\end{enumerate}
\end{question}


\section{Questions and further directions}\label{S:questions}

We conclude with a couple of questions not already mentioned above or elsewhere in the literature. As with our results in the preceding sections, the significance of these questions for the $\SRT^2_2$ vs.\ $\COH$ problem is methodological. There is still much about the interplay between combinatorics and computability in the construction of infinite homogeneous sets that we do not understand, but will almost certainly need to understand to find a solution to the problem. The questions and directions for further research below are thus aimed at enhancing this understanding.

Our first question concerns the problem of solving two instances of the pigeonhole principle in parallel. We recall the following terminology from the study of Weihrauch degrees. Given two problems $\Psf$ and $\Qsf$, the \emph{parallel product} of $\mathsf{P}$ and $\mathsf{Q}$ is the problem $\Psf \times \Qsf$ whose instances are pairs $(I, J)$ with $I$ a $\Psf$-instance and $J$ a $\Qsf$-instance, where a solution to such a pair $(I,J)$ is a pair $(X, Y)$ such that $X$ is a $\Psf$-solution to $I$ and $Y$ a $\Qsf$-solution to $J$.

\begin{question}
Is it the case that $\mathsf{D}^2_3 \cred \mathsf{D}^2_2 \times \mathsf{D}^2_2$?
\end{question}

\noindent Note that if $\times$ above is replaced by the \emph{compositional product} (see~\cite[Section 5]{BGP-TA}), then the answer above is yes (see, e.g.,~\cite[Section 4]{HJ-2016}).

The closest result we know in the direction of resolving this question is that $\D^2_3 \ncred \D^2_2$, which is due to Patey~\cite[Corollary 3.3]{Patey-2016}. The proof of this result is by a cardinality argument. For $j \leq k$, say a problem $\mathsf{P}$ \emph{preserves $j$ among $k$ hyperimmunities} if for every collection of hyperimmune functions $g_0,\ldots,g_{k-1}$, every instance $X$ of $\mathsf{P}$ has a solution $Y$ such that at least $j$ many of the $g_i$ are $Y$-hyperimmune. It is easy to see that this property is closed downwards under $\cred$, and the proof in~\cite{Patey-2016} shows that while $\mathsf{D}^2_2$ can preserve 2 among 3 hyperimmunities, $\mathsf{D}^2_3$ cannot. This argument will not work to settle the above question, because $\mathsf{D}^2_2 \times \mathsf{D}^2_2$ does not even preserve 2 among 4 hyperimmunities. To see this, consider a $4$-partition $A_0 \sqcup A_1 \sqcup A_2 \sqcup A_3 = \omega$ such that $\overline{A}_i$ is hyperimmune for each $i < 4$. Define the first $\mathsf{D}^2_2$-instance to be $A_0 \cup A_1$, and the second to be $A_0 \cup A_2$. Any solution to this pair as a $\mathsf{D}^2_2 \times \mathsf{D}^2_2$ instance will necessarily compute a function (namely, the principal function of either of its halves) that dominates at least three of the $g_i$. Note that the same argument can be made for $\D^2_4$, and indeed we have $\D^2_3 \cred \D^2_4$. A negative answer to our question would probably involve a variant of Mathias forcing with multiple reservoirs, since sharing the same reservoir would likely produce a solution to a given instance of $\mathsf{D}^2_2 \times \mathsf{D}^2_2$ as an instance of $\mathsf{D}^2_4$.

On a different note, there is still much we do not know about building $\Delta^0_2$ solutions to computable instances of $\SRT^2_2$/$\D^2_2$. For example, a longstanding open question is whether every such instance has a low$_2$ $\Delta^0_2$ solution (see, e.g.,~\cite[Question 6.46]{Hirschfeldt-2014}). We propose a new line of study.
The following definitions appear in several specific contexts in the literature. (See also~\cite[Chapter 12]{Patey-2016a}.) We state them here in complete generality since they seem like useful concepts in their own right. Given an instance $X$ of a problem $\Psf$, we write $\Psf(X)$ for the set of all its solutions, and $\operatorname{deg} \Psf(X)$ for the set of degrees of its solutions.

\begin{definition}
	Let $\Psf$ be a problem, $\Ccal$ a class of $\Psf$-instances, and $\dbf$ a Turing degree. We say:
	\begin{enumerate}
		\item $\Ccal$ is \emph{$\dbf$-bounding} (\emph{for $\mathsf{P}$}) if for every $\mathsf{P}$-instance $X$ of degree at most $\dbf$, there is an $\widehat{X} \in \Ccal$ such that every element of $\Psf(\widehat{X})$ computes a $\mathsf{P}$-solution to $X$.
		\item $\Ccal$ is a \emph{$\dbf$-basis} if it is $\dbf$-bounding and $(\forall X \in \Ccal)~\deg(X) \leq \dbf$.
		\item $\Ccal$ is a \emph{uniform $\dbf$-basis} if it is a $\dbf$-basis and there is a sequence $\seq{X_0,X_1,\ldots}$ of degree at most $\dbf$ such that $\Ccal = \{X_0,X_1,\ldots\}$.
	\end{enumerate}
\end{definition}

\begin{definition}
	Let $\Psf$ be a problem.
	\begin{enumerate}
		\item $\Psf$ \emph{admits a universal instance} if every degree $\dbf$ bounds the degree of a singleton $\dbf$-basis.
		\item $\Psf$ \emph{admits a uniform basis} if every degree $\dbf$ bounds the degree of a uniform $\dbf$-basis.
                  
	\end{enumerate}
\end{definition}

\reversemarginpar
Every problem trivially has a $\bf d$-basis, namely, the collection of all instances of the problem of degree at most $\bf d$. So the interest here is really in smaller bases, and in particular, in uniform ones. In the case of $\D^2_2$, we know the following. By relativizing and iterating the proof in Cholak, Jockusch, and Slaman~\cite[Theorem 3.7]{CJS-2001} that every computable $\D^2_2$ instance admits a low$_2$ solution, we can obtain, for any finite collection $\Ccal$ of computable $\D^2_2$ instances, a single low$_2$ degree bounding a solution to each instance in $\Ccal$. However, Mileti~\cite[Corollary 5.4.6]{Mileti-2004} has shown that there is no low${}_2$ degree bounding a solution to all computable $\mathsf{D}^2_2$ instances. It follows that $\D^2_2$ has no finite $\bf 0$-basis, or indeed, by relativizing this observation, a finite $\bf d$-basis, for any degree $\bf d$. In particular, $\D^2_2$ does not admit a universal instance. These facts motivate the following question:

\begin{question}
	Does $\mathsf{D}^2_2$ admit a uniform basis?
\end{question}


\begin{thebibliography}{10}

\bibitem{ADSS-2017}
Eric~P. Astor, Damir~D. Dzhafarov, Reed Solomon, and Jacob Suggs.
\newblock The uniform content of partial and linear orders.
\newblock {\em Ann. Pure Appl. Logic}, 168(6):1153--1171, 2017.

\bibitem{BGP-TA}
Vasco Brattka, Guido Gherardi, and Arno Pauly.
\newblock Weihrauch complexity in computable analysis.
\newblock To appear, arXiv:1707.03202.

\bibitem{BR-2017}
Vasco Brattka and Tahina Rakotoniaina.
\newblock On the uniform computational content of {R}amsey's theorem.
\newblock {\em J. Symbolic Logic}, 82(4):1278--1316, 2017.

\bibitem{CJS-2001}
Peter~A. Cholak, Carl~G. Jockusch, and Theodore~A. Slaman.
\newblock On the strength of {R}amsey's theorem for pairs.
\newblock {\em J. Symbolic Logic}, 66(1):1--55, 2001.

\bibitem{CLY-2010}
C.~T. Chong, Steffen Lempp, and Yue Yang.
\newblock On the role of the collection principle for {$\Sigma\sp 0\sb
  2$}-formulas in second-order reverse mathematics.
\newblock {\em Proc. Amer. Math. Soc.}, 138(3):1093--1100, 2010.

\bibitem{CSY-2014}
C.~T. Chong, Theodore~A. Slaman, and Yue Yang.
\newblock The metamathematics of stable {R}amsey's theorem for pairs.
\newblock {\em J. Amer. Math. Soc.}, 27(3):863--892, 2014.

\bibitem{DDHMS-2016}
Fran{\c{c}}ois~G. Dorais, Damir~D. Dzhafarov, Jeffry~L. Hirst, Joseph~R.
  Mileti, and Paul Shafer.
\newblock On uniform relationships between combinatorial problems.
\newblock {\em Trans. Amer. Math. Soc.}, 368(2):1321--1359, 2016.

\bibitem{DH-2010}
Rodney~G. Downey and Denis~R. Hirschfeldt.
\newblock {\em Algorithmic randomness and complexity}.
\newblock Theory and Applications of Computability. Springer, New York, 2010.

\bibitem{DHLS-2001}
Rodney~G. Downey, Denis~R. Hirschfeldt, Steffen Lempp, and Reed Solomon.
\newblock A {$\Delta\sp 0\sb 2$} set with no infinite low subset in either it
  or its complement.
\newblock {\em J. Symbolic Logic}, 66(3):1371--1381, 2001.

\bibitem{Dzhafarov-2015}
Damir~D. Dzhafarov.
\newblock Cohesive avoidance and strong reductions.
\newblock {\em Proc. Amer. Math. Soc.}, 143(2):869--876, 2015.

\bibitem{Dzhafarov-2016}
Damir~D. Dzhafarov.
\newblock Strong reductions between combinatorial principles.
\newblock {\em J. Symbolic Logic}, 81(4):1405--1431, 2016.

\bibitem{DJ-2009}
Damir~D. Dzhafarov and Carl~G. Jockusch, Jr.
\newblock {R}amsey's theorem and cone avoidance.
\newblock {\em J. Symbolic Logic}, 74(2):557--578, 2009.

\bibitem{DP-2017}
Damir~D. Dzhafarov and Ludovic Patey.
\newblock Coloring trees in reverse mathematics.
\newblock {\em Adv. Math.}, 318:497--514, 2017.

\bibitem{DPSW-2017}
Damir~D. Dzhafarov, Ludovic Patey, Reed Solomon, and Linda~Brown Westrick.
\newblock Ramsey's theorem for singletons and strong computable reducibility.
\newblock {\em Proc. Amer. Math. Soc.}, 145(3):1343--1355, 2017.

\bibitem{Flood-2012}
Stephen Flood.
\newblock Reverse mathematics and a {R}amsey-type {K\"o}nig's lemma.
\newblock {\em J. Symbolic Logic}, 77(4):1272--1280, 2012.

\bibitem{Hirschfeldt-2014}
Denis~R. Hirschfeldt.
\newblock {\em Slicing the Truth: On the Computable and Reverse Mathematics of
  Combinatorial Principles}.
\newblock Lecture Notes Series. Institute for Mathematical Sciences. National
  University of Singapore. World Scientific Publishing Co. Pte. Ltd.,
  Hackensack, NJ, 2014.

\bibitem{HJ-2016}
Denis~R. Hirschfeldt and Carl~G. Jockusch, Jr.
\newblock On notions of computability-theoretic reduction between {$\Pi_2^1$}
  principles.
\newblock {\em J. Math. Log.}, 16(1):1650002, 2016.

\bibitem{HJKLS-2008}
Denis~R. Hirschfeldt, Carl~G. Jockusch, Jr., Bj{\o}rn Kjos-Hanssen, Steffen
  Lempp, and Theodore~A. Slaman.
\newblock The strength of some combinatorial principles related to {R}amsey's
  theorem for pairs.
\newblock In {\em Computational prospects of infinity. {P}art {II}. {P}resented
  talks}, volume~15 of {\em Lecture Notes Series. Institute for Mathematical
  Sciences. National University of Singapore}, pages 143--161. World Scientific
  Publishing Co. Pte. Ltd., Hackensack, NJ, 2008.

\bibitem{HS-2007}
Denis~R. Hirschfeldt and Richard~A. Shore.
\newblock Combinatorial principles weaker than {R}amsey's theorem for pairs.
\newblock {\em J. Symbolic Logic}, 72(1):171--206, 2007.

\bibitem{HM-2017}
Jeffry~L. Hirst and Carl Mummert.
\newblock Reverse mathematics of matroids.
\newblock In {\em Computability and {C}omplexity}, volume 10010 of {\em Lecture
  Notes in Comput. Sci.}, pages 143--159. Springer, Cham, 2017.

\bibitem{HM-TA}
Jeffry~L. Hirst and Carl Mummert.
\newblock Using {R}amsey's theorem once.
\newblock {\em Arch. Math. Logic}, 58(7-8):857--866, 2019.

\bibitem{JS-1993}
Carl Jockusch and Frank Stephan.
\newblock A cohesive set which is not high.
\newblock {\em Math. Logic Quart.}, 39(4):515--530, 1993.

\bibitem{LST-2013}
Manuel Lerman, Reed Solomon, and Henry Towsner.
\newblock Separating principles below {R}amsey's theorem for pairs.
\newblock {\em J. Math. Log.}, 13(2):1350007, 2013.

\bibitem{Liu-2012}
Jiayi Liu.
\newblock {{RT}$^2_2$ does not imply {WKL}$_0$}.
\newblock {\em J. Symbolic Logic}, 77(2):609--620, 2012.

\bibitem{Mileti-2004}
Joseph~R. Mileti.
\newblock {\em Partition Theorems and Computability Theory}.
\newblock PhD thesis, University of Illinois at Urbana-Champaign, 2004.

\bibitem{MP-TA2}
Benoit Monin and Ludovic Patey.
\newblock {SRT}$^2_2$ does not imply {COH} in $\omega$-models.
\newblock To appear, arXiv:1905.08427.

\bibitem{MP-TA}
Benoit Monin and Ludovic Patey.
\newblock {$\Pi^0_1$}-encodability and omniscient reductions.
\newblock {\em Notre Dame J. Form. Log.}, 60(1):1--12, 2019.

\bibitem{Nichols-TA}
David Nichols.
\newblock Strong reductions between relatives of the stable {R}amsey's theorem.
\newblock To appear, arXiv:1711.06532.

\bibitem{Patey-TA4}
Ludovic Patey.
\newblock {Somewhere over the rainbow {R}amsey theorem for pairs}.
\newblock To appear, arXiv:1501.07424.

\bibitem{Patey-2016d}
Ludovic Patey.
\newblock Open questions about {R}amsey-type statements in reverse mathematics.
\newblock {\em Bull. Symb. Log.}, 22(2):151--169, 2016.

\bibitem{Patey-2016c}
Ludovic Patey.
\newblock Partial orders and immunity in reverse mathematics.
\newblock In {\em Pursuit of the {U}niversal}, volume 9709 of {\em Lecture
  Notes in Comput. Sci.}, pages 353--363. Springer, Cham, 2016.

\bibitem{Patey-2016b}
Ludovic Patey.
\newblock The strength of the tree theorem for pairs in reverse mathematics.
\newblock {\em J. Symbolic Logic}, 81(4):1481--1499, 2016.

\bibitem{Patey-2016a}
Ludovic Patey.
\newblock {\em {The Reverse Mathematics of Ramsey-type Theorems}}.
\newblock PhD thesis, {Universit{\'e} Paris Diderot (Paris 7) Sorbonne Paris
  Cit{\'e}}, February 2016.

\bibitem{Patey-2016}
Ludovic Patey.
\newblock The weakness of being cohesive, thin or free in reverse mathematics.
\newblock {\em Israel J. Math.}, 216(2):905--955, 2016.

\bibitem{Patey-2017b}
Ludovic Patey.
\newblock Iterative forcing and hyperimmunity in reverse mathematics.
\newblock {\em Computability}, 6(3):209--221, 2017.

\bibitem{SS-1995}
David Seetapun and Theodore~A. Slaman.
\newblock On the strength of {R}amsey's theorem.
\newblock {\em Notre Dame J. Formal Logic}, 36(4):570--582, 1995.
\newblock Special Issue: Models of arithmetic.

\bibitem{Shore-2016}
Richard~A. Shore.
\newblock The {T}uring degrees: an introduction.
\newblock In {\em Forcing, {I}terated {U}ltrapowers, and {T}uring {D}egrees},
  volume~29 of {\em Lect. Notes Ser. Inst. Math. Sci. Natl. Univ. Singap.},
  pages 39--121. World Scientific Publishing Co. Pte. Ltd., Hackensack, NJ,
  2016.

\bibitem{Simpson-2009}
Stephen~G. Simpson.
\newblock {\em Subsystems of {S}econd {O}rder {A}rithmetic}.
\newblock Perspectives in Logic. Cambridge University Press, Cambridge, second
  edition, 2009.

\bibitem{Soare-2016}
Robert~I. Soare.
\newblock {\em Turing Computability: Theory and Applications}.
\newblock Springer-Verlag, Berlin, 2016.

\bibitem{Wang-2016}
Wei Wang.
\newblock The definability strength of combinatorial principles.
\newblock {\em J. Symbolic Logic}, 81(4):1531--1554, 2016.

\bibitem{Weihrauch-1992}
K.~Weihrauch.
\newblock The degrees of discontinuity of some translators between
  representations of the real numbers.
\newblock Technical report {TR}-92-050, International Computer Science
  Institute, Berkeley, 1992.

\end{thebibliography}
\end{document}